\documentclass[12pt,a4paper]{amsart}
\usepackage{amsmath,amssymb,amsfonts, float}
\usepackage[all]{xy}
\usepackage{caption}
\usepackage{enumerate}
\usepackage[shortlabels]{enumitem}
\usepackage{mathpazo}
\usepackage{a4wide}

\usepackage{bm}
\usepackage{graphicx}
\usepackage[breaklinks=true]{hyperref}

\usepackage{xcolor}
\usepackage{multicol}
\usepackage{tabularx}

\usepackage{tikz}
\usepackage{tikz-cd}

\usepackage{hyperref}
\usepackage[capitalize]{cleveref}

\usepackage[normalem]{ulem}

\newcommand{\tung}[1]{{\color{red}#1}}
\newcommand{\Tan}[1]{{\color{brown}#1}}
\newtheorem{thm}{Theorem}[section] 
\newtheorem*{thm*}{Theorem} 
\newtheorem{prop}[thm]{Proposition}
\newtheorem{lem}[thm]{Lemma}
\newtheorem{cor}[thm]{Corollary}

\theoremstyle{definition}
\newtheorem{definition}[thm]{Definition}
\newtheorem{expl}[thm]{Example}
\newtheorem{conj}[thm]{Conjecture}
\newtheorem{algo}[thm]{Algorithm}

\newtheorem{question}[thm]{Question}
\newtheorem{rem}[thm]{Remark}

\newtheorem{speculation}[thm]{Speculation}

\DeclareMathOperator{\R}{\mathbb{R}}
\DeclareMathOperator{\C}{\mathbb{C}}
\DeclareMathOperator{\Z}{\mathbb{Z}}
\DeclareMathOperator{\N}{\mathbb{N}}
\DeclareMathOperator{\F}{\mathbb{F}}
\DeclareMathOperator{\Fbar}{\overline{\mathbb{F}}}
\DeclareMathOperator{\Q}{\mathbb{Q}}

\DeclareMathOperator{\Gal}{\text{Gal}}

\DeclareMathOperator{\Res}{{\rm Res}}

\DeclareMathOperator{\disc}{{\rm disc}}
\DeclareMathOperator{\sgn}{\rm sgn}
\DeclareMathOperator{\Frob}{\rm Frob}

    \DeclareFontFamily{U}{wncy}{}
    \DeclareFontShape{U}{wncy}{m}{n}{<->wncyr10}{}
    \DeclareSymbolFont{mcy}{U}{wncy}{m}{n}
    \DeclareMathSymbol{\Sha}{\mathord}{mcy}{"58}

\numberwithin{equation}{section}

\newcommand{\fbar}{\overline{f}}
\newcommand{\gbar}{\overline{g}}
\newcommand{\hb}{\overline{h}}

\newcommand{\rev}{{\text{rev}}}

\DeclareSymbolFont{bbold}{U}{bbold}{m}{n}
\DeclareSymbolFontAlphabet{\mathbbold}{bbold}

\newcommand{\mult}{{\rm mult}}
\newcommand{\Char}{\mathrm{char}}
\newcommand{\legendre}[2]{\ensuremath{\left( \frac{#1}{#2} \right) }}
\newcommand{\abs}[1]{\left\lvert#1\right\rvert}

\newcommand{\SC}[1]{\noindent {\color{blue!80!black}[SC] #1 }}
\newcommand{\TN}[1]{\noindent {\color{red!80!black}[Tung] #1 }}

\begin{document}
\title{Fekete polynomials of principal Dirichlet characters}
 \author{Shiva Chidambaram, J\'an Min\'a\v{c}, Tung T. Nguyen, Nguy$\tilde{\text{\^{E}}}$n Duy T\^{a}n }
\address{Department of Mathematics, Massachusetts Institute of Technology, Massachusetts Avenue
Cambridge, MA 02139-4307}
\email{shivac@mit.edu}
\address{Department of Mathematics, Western University, London, Ontario, Canada N6A 5B7}
\email{minac@uwo.ca}
\date{\today}

 \address{Department of Mathematics, Western University, London, Ontario, Canada N6A 5B7}
 \email{tungnt@uchicago.edu}
 
  \address{
 School of Applied  Mathematics and 	Informatics, Hanoi University of Science and Technology, 1 Dai Co Viet Road, Hanoi, Vietnam } 
\email{tan.nguyenduy@hust.edu.vn}
 
\thanks{SC was supported by Simons Foundation grant 550033. JM is partially supported  by the Natural Sciences and Engineering Research Council of Canada (NSERC) grant R0370A01. He gratefully acknowledges the Western University Faculty of Science Distinguished Professorship 2020-2021. JM and TN  acknowledge the support of the Western Academy for Advanced Research. NDT is funded by Vingroup Joint Stock Company and supported by Vingroup Innovation Foundation (VinIF) under the project code VINIF.2021.DA00030}

\maketitle
\begin{abstract}

Fekete polynomials associated to quadratic Dirichlet characters have interesting arithmetic properties, and have been studied in many works. In this paper, we study a seemingly simpler yet rich variant: the Fekete polynomial $F_n(x) = \sum_{a=1}^n \chi_n(a) x^a$ associated to a principal Dirichlet character $\chi_n$ of modulus $n$. We investigate the cyclotomic factors of $F_n$ and conjecturally describe all of them. One interesting observation from our computations is that the non-cyclotomic part $f_n$ of $F_n(x)/x$ seems to be always irreducible. We study this factor closely in the special case that $n$ is a product of two odd primes, proving separability in specific cases, and studying its coefficients and special values. Combining these theoretical results with computational evidence lets us identify the Galois group of $f_n$ for small $n$, and raises precise questions in general.

\end{abstract}
\tableofcontents

\section{Introduction}
Let $\chi: (\Z/n)^{\times} \to \C^{\times}$ be a Dirichlet character with modulus $n>1.$ Associated to $\chi$ is the Dirichlet $L$-function defined by 
\[ L(\chi,s)= \sum_{a=1}^{\infty} \frac{\chi(a)}{a^s}.\] 
This infinite series is absolutely convergent when $\Re(s)>1$. The function $L(\chi, s)$ has a meromorphic continuation to the entire complex plane, with a simple pole at $s=1$ in the case $\chi$ is a principal character.  Furthermore, $L(\chi, s)$ has the following integral representation (see \cite[Proposition 3.3]{MTT4}. The character $\chi$ is assumed to be primitive in \cite{MTT4}, but it is not necessary and the proof goes through without this assumption)
\begin{equation} \label{eq:integral}
\Gamma(s) L(\chi, s)=  \int_{0}^1 \frac{(-\log t)^{s-1} }{t} \frac{F_{\chi}(t)}{1-t^n} dt,
\end{equation}
where $\Gamma(s)$ is the Gamma function and 
\[ F_{\chi}(x) =\sum_{a=0}^{n-1} \chi(a)x^a .\] 
While studying quadratic characters $\chi = \left(\frac{\cdot}{p} \right)$ of prime conductor $p$, Fekete  made the observation that if $F_{\chi}(x)$ has no real zeroes in the interval $0<x<1$, then $L(s,\chi)$ has no real zeroes on $(0,\infty)$. In this sense, the study of $F_{\chi}(x)$ could shed light on the existence of  Siegel zeroes near $s = 1$. For this historical reason, the $F_{\chi}(x)$ are called Fekete polynomials.

Fekete polynomials have a rich mathematical history. They appear in Gauss's sixth proof of the quadratic reciprocity law (see \cite[Chapter 10, Section 3]{lemmermeyer}). Various aspects of Fekete polynomials have been studied over the years, including their extremal properties (see \cite{borwein2002computational, borwein2001extremal}), Mahler measure (see \cite{erdelyi2018improved, klurman2023l_q}), connections to oscillations of quadratic $L$-functions (see \cite{baker1990oscillations}), and distribution of their complex roots (\cite{conrey2000zeros}).

In recent works (\cite{MTT4, MTT3}), the last named three authors have analyzed the arithmetic properties of Fekete polynomials when $\chi$ is a primitive quadratic Dirichlet character. These works have shown that Fekete polynomials contain valuable arithmetic information, such as class number and orders of $K$-groups of certain quadratic fields. Furthermore, extensive computational evidence suggests that $F_{\chi}(x)/x$ has exactly one irreducible non-cyclotomic factor $f_{\chi}$, and that the Galois group of $f_{\chi}$ is as large as possible, as stated in \cite[Conjecture 4.9, Conjecture 4.13]{MTT3} and \cite[Conjecture 11.16]{MTT4}.

In this article, we consider the story for principal Dirichlet characters. Already in this seemingly simple situation, we notice some interesting phenomena akin to the primitive quadratic case discussed above. Concretely, we consider $\chi_n : \Z \to \C^{\times}$ defined by
\[ \chi_n(a)= \begin{cases}0&{\text{if }}\;\gcd(a,n)>1\\ 1 &{\text{if }}\;\gcd(a,n)=1.\end{cases} \]
and let $F_{n}(x)$ denote the associated Fekete polynomial $F_{\chi_n}(x)$. That is,\\
\begin{align}
\label{def:Fekete}
    F_n(x)= \sum_{\substack{0 \leq a \leq n-1\\ \gcd(a,n)=1}} x^a.
\end{align}
Since $\gcd(n-1,n)=1$, the degree of $F_n$ is $n-1.$ In this article, we begin the study of $F_n$, focusing in particular on the determination and arithmetic of the factors of $F_n$, both cyclotomic and non-cyclotomic.

We observe that the non-cyclotomic part of $F_n(x)/x$, which we denote by $f_n$, seems to be always irreducible. Similar to the case of quadratic characters investigated in \cite{MTT4, MTT3}, the Galois group of $f_n$ is as large as possible subject to a condition on the discriminant. We also notice that the coefficients of $f_n$ are relatively small. When $n=3p$ for a prime $p > 3$, we prove in \cref{thm:coeff_3p} that the coefficients of $f_{3p}$ belong to the set $\{-2, -1, 0, 1, 2 \}$. This suggests that $f_n$ may have noteworthy extremal properties that we intend to explore in the future. It is important to note that we approach this project from a computational standpoint, meaning that we discovered many results in our article by computing and analyzing a large dataset. The codes and data are available in the GitHub repository \cite{fekete}.

We remark that the theory of Fekete polynomials is closely related to the construction of certain Paley graphs. When $\chi$ is a primitive quadratic Dirichlet character, this connection is discussed in \cite{paleygraph}. When $\chi$ is a principal Dirichlet character, the corresponding Paley graph is called a unitary Caley graph (see \cite{bavsic2015polynomials, klotz2007some}). These types of Paley graphs have found applications in various fields such as coding and cryptography theory (see \cite{ghinelli2011codes, javelle2014cryptographie}). We hope that this work sheds more light on further applications of Fekete polynomials and Paley graphs.

The article is structured as follows.
In \cref{sec:cyclotomic_factors}, we present our main results concerning the factors of $F_n$. Firstly in \cref{subsec:reduction_to_squarefree}, we show that restricting to squarefree $n$ does not miss any interesting phenomena. We then proceed to describe in \cref{thm:cyclotomic_factor_3} a large source of cyclotomic factors of $F_n$ and prove in \cref{thm:multiplicity_of_Phid} that all cyclotomic factors are simple factors, with the exception of the $4$-th cyclotomic polynomial $\Phi_4(x) = x^2+1$ which occurs with multiplicity $2$ when $n$ is even.
We also make a conjecture relating to the full description of cyclotomic factors of $F_n$ (see \cref{conj:equidistribution}).
\cref{sec:pq} is devoted to the case where $n=pq$ for distinct odd primes $p, q$. We conjecturally determine all cyclotomic factors of $F_n$ in this special case (see \cref{thm:cyclotomic_factors_pq1}).
Based on this, we define the Fekete polynomial $f_n$ and its trace polynomial $g_n$.
We then study some arithmetic properties of $F_n$ in characteristic $p$ with the goal of showing that $f_n$ is separable. In \cref{sec:3p}, we focus on the case $n=3p$. We show in \cref{thm:coeff_3p} that the coefficients of $f_{3p}$ lie in the set $\{-2,-1,0,1,2\}$.
We also prove that $f_{3p}$ is separable, by first proving it in characteristic $p$.
In \cref{sec:5p} we consider the case $n = 5p$, once again establishing separability of $f_n$, although the proof is more involved.
The last two sections discuss computational techniques and results.
\cref{sec:irreducibility} discusses algorithms for checking irreducibility of $f_n$ and $g_n$, and \cref{sec:Galois} investigates their Galois groups. Finally \cref{sec:zeros_on_unitcircle} discusses a tangential question of an analytic flavor, about the proportion of complex zeroes of $F_n$ on the unit circle, using the approach of \cite{conrey2000zeros}.

\section{Cyclotomic factors of $F_n$ and their multiplicity}
\label{sec:cyclotomic_factors}
\subsection{Reduction to the squarefree case}
\label{subsec:reduction_to_squarefree}
Let $n$ be a positive integer and $n_0$ the radical of $n$ which is defined as the product of the distinct prime divisors of $n$. Let $\chi_n$ and $\chi_{n_0}$ be the principal Dirichlet characters associated with $n$ and $n_0$ as explained in the introduction. For an integer $a$, we know that $\gcd(a,n)=1$ if and only if $\gcd(a,n_0)=1.$ Therefore,  by definition, we see that $L(\chi_n,s)=L(\chi_{n_0}, s).$ By the integral representations of these $L$-functions \cref{eq:integral} we conclude that for all $s>1$

\[ \int_{0}^1 \frac{(-\log(t))^{s-1} }{t} \frac{F_{n}(t)}{1-t^{n}} dt = \int_{0}^1 \frac{(-\log(t))^{s-1} }{t} \frac{F_{n_0}(t)}{1-t^{n_0}} dt .\] 
This suggests the following proposition. 
\begin{prop} \label{prop:squarefree_condition}
Let $n$ be an integer and $n_0$ the radical of $n$. Then we have the following equality 
\[  \frac{F_{n_0}(x)}{1-x^{n_0}} = \frac{F_{n}(x)}{1-x^n}. \]
\end{prop}

\begin{proof}

Since $n_0$ is the radical of $n$, we have $\gcd(a,n)=1$ if and only if $\gcd(a, n_0)=1$. Therefore, we have 
\[ \frac{F_{n_0}(x)}{1-x^{n_0}} = F_{n_0}(x) \sum_{k=0}^{\infty} x^{kn_0} = \sum_{\substack{1 \leq a\\ \gcd(a,n_0)=1}} x^a = \sum_{\substack{1 \leq a\\ \gcd(a,n)=1}} x^a = \frac{F_n(x)}{1-x^n}. \] 
\end{proof}

\begin{cor}
Let $f \in \Z[x]$ be a non-cyclotomic irreducible polynomial. Then $f$ is a divisor of $F_{n}$ if and only if $f$ is a divisor of $F_{n_0}$. \qed
\end{cor}

\begin{cor} \label{cor:derivative_negative_one}
Suppose that $n$ is an odd integer and $n_0$ is its radical. Then $F_n(-1)=F_{n_0}(-1)=0$ and
$ F_{n}'(-1) = F_{n_0}'(-1)$. \qed
\end{cor}

\subsection{Cyclotomic factors of $F_n$ }

We let $n$ be a positive squarefree integer here onwards, thanks to \cref{prop:squarefree_condition}. To get started we calculate the Fekete polynomial $F_n(x)$ in a few simple situations.
\begin{expl}
    \label{expl:p_and_pk}
    Let $p$ be a prime number. 
    Then we have the factorisation of $F_p(x)$.
    \begin{align}
        \label{eqn:Factorize_Fp}
        F_p(x) = x+x^2+\ldots+x^{p-1} = x \frac{1-x^{p-1}}{x-1} =x \prod_{\substack{d|p-1\\ d>1}} \Phi_d(x).
    \end{align}
\end{expl}

\begin{expl}
    \label{expl:2p}
    Let $p$ be an odd prime. Then we have the following factorisation of $F_{2p}(x)$.
    \begin{align}
    \begin{split}
        F_{2p}(x)&= (x+x^3+\cdots+x^{p-2})+(x^{p+2}+x^{p+4}+\cdots+x^{2p-1})\\
        &=x(1+x^2+\cdots+x^{p-3})(1+x^{p+1})=x\dfrac{x^{p-1}-1}{x^2-1}\dfrac{x^{2(p+1)}-1}{x^{p+1}-1}\\
        &=x\prod_{2<d\mid (p-1)}\Phi_d(x)\prod_{\substack{d\mid 2(p+1)\\d\nmid p+1}}\Phi_d(x).
    \end{split}
    \end{align}
\end{expl}
While we were able to completely describe the cyclotomic factors of $F_n$ in the two examples above, this is quite delicate in general. Moreover, these examples are unrepresentative of the general situation where non-cyclotomic factors also show up. Using SageMath we computed the factorisation of the Fekete polynomial $F_n$ as defined in \cref{def:Fekete} for all $n < 10^4$. When $n \neq p$ or $2p$ for any prime $p$, we observe that $F_n(x)/x$ has many cyclotomic factors and \emph{exactly one} irreducible non-cyclotomic factor. This striking observation leads us to the study of the cyclotomic factors of $F_n$.

Our starting point in this endeavor is the result that if $d|n$ then the $d$-th cyclotomic polynomial $\Phi_d$ is \emph{not} a factor of $F_n$. This is a direct consequence of the theory of Ramanujan sums (see \cite{hardy1979introduction} for a detailed treatment) which is partially summarized in the following proposition. 

\begin{prop} \label{prop:value_at_zeta_d}
Let $n$ be a positive integer and $d$ a divisor of $n$. Let $k$ be  a field such that $d$ is invertible in $k$. Suppose $\zeta_d$ is a primitive $d$-th root of unity in $k$.  Then 
\[ F_n(\zeta_d) = \frac{\mu(d) \varphi(n)}{\varphi(d)} .\]
Here $\mu$ denotes the Mobius function and we consider $F_n[x]$ as a polynomial over $k[x]$ under the canonical map $\Z[x] \to k[x].$

\end{prop}
\begin{proof}
If $\Char(k) = 0$, by embedding the subfield $\Q(\zeta_d) \subseteq k$ in $\mathbf{C}$, we have
\[ F_n(\zeta_d) = \sum_{\substack{1 \leq a \leq n \\ \gcd(a,n)=1}} \exp\left(\frac{2\pi i a}{d}\right),\] which is a Ramanujan sum (see \cite[Section 5.6]{hardy1979introduction}).
Hence by \cite[Theorem 272]{hardy1979introduction}, we have $F_n(\zeta_d) =  \frac{\mu(d) \varphi(n)}{\varphi(d)}$.

To deal with positive characteristics, we first note that the statement is entirely algebraic with all objects defined over the ring of integers $\mathcal{O}_F = Z[\zeta_d]$ of the cyclotomic field $F = \Q(\zeta_d)$. Indeed, we have
\[ F_n(x) \in \Z[x],\quad \zeta_d \in \mathcal{O}_F,\quad \mu(d) = \sum_{\substack{1 \leq a \leq n \\ \gcd(a,n)=1}} \zeta_d^a \in \mathcal{O}_F,\quad \phi(n) = \sum_{\substack{1 \leq a \leq n \\ \gcd(a,n)=1}} 1 \in \Z. \]
If $\Char(k) = p$, by considering reduction modulo a prime ideal of $\mathcal{O}_F$ above $p$, we see that the same formula should hold over $k$.
\end{proof}

\begin{cor}
\label{cor:d_divisor_n}
If $d|n$ then $F_n(\zeta_d) \neq 0.$ In other words, $\Phi_d$ is not a factor of $F_n.$ \qed
\end{cor}

Hereafter we assume that $d \nmid n$. Let $d_1 = \gcd(d,n)$ so that $d_1 \neq d$, and let $S$ be the set $\{ 0 \leq a < d | \gcd(a,d_1) = 1\}$. Note that if we identify the integer residues modulo $d$ with the set $[d] = \{0, 1, \ldots, d-1\}$, then $S \subset [d]$ is the preimage of the set $(\Z/d_1)^{\times}$ under the reduction map $\Z/d \to \Z/d_1$. Thus we have $\#S = \frac{\varphi(d_1)d}{d_1}$. Consider the polynomial $F_S(x) = \sum_{s \in S} x^s$. The following lemma gives some useful information about $F_S$.
\begin{lem}
\label{lem:F_S}
    With notation as above, we have $F_S(x) = \frac{1-x^d}{1-x^{d_1}} F_{d_1}(x)$.
    Let $m|d$ be a positive integer, and $\zeta = \zeta_m$ be a primitive $m$-th root of unity. Then 
    \[ F_S(\zeta)= \begin{cases}0 &{\text{if }}\; m \nmid d_1 \\ \frac{d}{d_1} \frac{\mu(m) \varphi(d_1)}{\varphi(m)} &{\text{if }}\; m|d_1 .\end{cases} \]
\end{lem}
\begin{proof}
    Noting that the expression $(1-x^{d_1}) F_S(x)$ is a telescopic sum, we get
    \[ (1-x^{d_1}) F_S(x) = (1-x^{d_1}) \sum_{0 \leq j < d/d_1} \sum_{\substack{1 \leq i \leq d_1 \\ (i,d_1) = 1}} x^{jd_1+i} = (1-x^d) F_{d_1}(x).\]
    Since $m | d$ we have $1-\zeta^d = 0$. If $m \nmid d_1$, then $1-\zeta^{d_1} \neq 0$. So $F_S(\zeta) = 0$.
    If $m | d_1$, then the value of $\frac{1-x^d}{1-x^{d_1}}$ at $x = \zeta$ is $\frac{d}{d_1}$. Further \cref{prop:value_at_zeta_d} gives $F_{d_1}(\zeta) = \frac{\mu(m)\varphi(d_1)}{\varphi(m)}$.
\end{proof}

We can now state a sufficient condition for the $d$-th cyclotomic polynomial $\Phi_d$ to be a factor of $F_n$. Note that, in order to determine whether $F_n(\zeta_d)$ is zero or not, it is sufficient to consider $F_n(x) \pmod{x^d - 1}$.

\begin{prop}
\label{prop:equidistribution_condition_new}
    Let $n$ be a positive squarefree integer. Let $d > 1$ be an integer not dividing $n$. Let $d_1 = \gcd(d,n)$ and $S = \{ 1 \leq a \leq d | \gcd(a,d_1) = 1\}$. If the elements of $\{1 \leq a \leq n | \gcd(a,n) = 1\}$ when reduced modulo $d$, equidistribute among the elements of $S$, then the $d$-th cyclotomic polynomial $\Phi_d$ divides $F_n$.
\end{prop}
\begin{proof}
    Let $F_S(x) = \sum_{s \in S} x^s$ be the polynomial introduced earlier. If the equidistribution condition holds, then by reducing modulo $d$ the exponents of monomials in $F_n$, we see that $F_n(x) \equiv c F_S(x) \pmod{x^d-1}$ for some constant $c$. To be precise, $c = c_{n,d} = \frac{\varphi(n)}{\#S} = \frac{\varphi(n)d_1}{\varphi(d_1)d}$.
    So it is enough to show that $F_S(\zeta_d) = 0$, which we get from \cref{lem:F_S} since $d > d_1$.
\end{proof}

\begin{rem}
\label{rem:equivalence_equidistribution}
    In the proof above, we saw that the equidistribution property implies $F_n \equiv c F_S \pmod{x^d-1}$. In fact they are equivalent. If $F_n = c F_S + (x^d-1) G$ for some $G \in \Q[x]$, then for each $s \in S$, we can sum the coefficients of the monomials $x^a$ with $a \equiv s \pmod d$ occuring on each side. This gives us that $\#\{1 \leq a \leq n | \gcd(a,n) = 1, a \equiv s \pmod d\} = c$ independent of $s$, yielding equidistribution.
\end{rem}

    It is interesting to ask if the converse of \cref{prop:equidistribution_condition_new} is true. 
    In all cases we have computed, the converse holds. See \cref{thm:equidistribution_provedcases} for the cases where the converse is proven to be true. This prompts us to make the conjecture.

\begin{conj}
\label{conj:equidistribution}
    Let $n, d, d_1, S$ be as in \cref{prop:equidistribution_condition_new}. Then $\Phi_d$ divides $F_n$ if and only if the elements of $\{1 \leq a \leq n | \gcd(a,n) = 1\}$ when reduced modulo $d$, equidistribute in $S$. Equivalently, $\Phi_d$ divides $F_n$ if and only if $x^d-1$ divides $F_n - \frac{\varphi(n)d_1}{\varphi(d_1)d} F_S$
\end{conj}

\begin{rem}
    In \cref{sec:pq} we give a complete description of the cyclotomic factors of $F_n$ when $n = pq$ is a product of two odd primes, contingent on \cref{conj:equidistribution}.
\end{rem}

\begin{rem}
\label{rem:equidistribute_prime_d}
    If $d = p$ is a prime not dividing $n$, then it is straightforward to see that \cref{conj:equidistribution} holds. This is because the only linear relation among $\{1, \zeta_p, \zeta_p^2, \ldots, \zeta_p^{p-1}\}$ is that they sum to $0$.
\end{rem}

We now write down a recursive formula for $F_n$. This will be useful in obtaining a neater description of $F_n$ in \cref{prop:form_of_F_n}, which will be crucial in describing certain explicit cyclotomic factors of $F_n$ (\cref{thm:cyclotomic_factor_3}) and their multiplicites (\cref{thm:multiplicity_of_Phid}).
\begin{prop} \label{prop:recursive_F}
    Let $p$ be a prime number such that $\gcd(p,n)=1.$ Then we have the following recursive formula 
    \[ \frac{F_{np}(x)}{1-x^{np}} = \dfrac{F_{n}(x)}{1-x^{n}} - \frac{F_{n}(x^p)}{1-x^{np}}. \]
\end{prop}
\begin{proof}
We have 
\[ \begin{aligned}
\dfrac{F_{np}(x)}{1-x^{np}}& =\sum_{\substack{1\leq a\\ \gcd(a,np)=1}}x^a
=\sum_{\substack{1\leq a\\ \gcd(a,n)=1}}x^a- \sum_{\substack{1\leq a \\\gcd(a,n)=1}}x^{pa}
= \dfrac{F_{n}(x)}{1-x^n}- \dfrac{F_{n}(x^p)}{1-x^{np}}.
\end{aligned}
\qedhere\]
\end{proof}
\begin{cor}
    Let $n>1$ be a squarefree integer and $d$ a positive integer not dividing $n$. If $\Phi_d$ is a  factor of $F_n$, then $\Phi_d$ is also a factor of $F_{np}$ for every prime number $p$ with ${\rm gcd}(d,p)=1$.
\end{cor}
\begin{proof}
    If $\zeta$ is a primitive $d$-th root of unity, $\zeta^p$ is also a primitive $d$-th root of unity. Hence $F_n(\zeta)=F_n(\zeta^p)=0$. Therefore, \cref{prop:recursive_F} implies $F_{np}(\zeta)=0$ as well. 
\end{proof}
\begin{cor}
    Let $n>1$ be a squarefree integer and $d$ a positive integer not dividing $n$. Let $p,q$ be two distinct primes such that 
$p \equiv q \pmod{d}$. Then $\Phi_d$ is a factor of $F_{np}$ if and only if it is a factor of $F_{nq}.$
\end{cor}
\begin{proof}
    Since $p \neq q$ and $p \equiv q \pmod d$, we get that $p \nmid d$. So $d \nmid np$ and we can evalute the formula in \cref{prop:recursive_F} at $\zeta_d$. Since $\zeta_d^p = \zeta_d^q$, we deduce that $F_{np}(\zeta_d) = F_{nq}(\zeta_d)$.
\end{proof}
Using this recursive formula of \cref{prop:recursive_F}, we get the following description for $F_n$ as a sum over divisors of $n$. This will used crucially in what follows.
\begin{prop}
\label{prop:form_of_F_n}
    Let $n>1$ be a squarefree integer.  Then
    \begin{align}
    \label{eq:form_of_Fn}
        \frac{F_n(x)}{1-x^n} = \sum_{m|n} \mu(m) \frac{x^m}{1-x^m}.
    \end{align}
\end{prop}
\begin{proof} We prove by induction on the number of prime factors of $n$. If $n=q$ is prime then
\[
\frac{F_q(x)}{1-x^q} = \frac{x+x^2+\cdots+x^{q}-x^q}{1-x^q} = \dfrac{x}{1-x}-\dfrac{x^q}{1-x^q} = \sum_{m|q} \frac{\mu(m) x^m}{1-x^m}.
\]
Now suppose that the statement is true for a squarefree integer $n$. Let $p$ be a prime number such that $\gcd(p,n)=1$. We show that the statement holds true for $np$. Indeed, by \cref{prop:recursive_F} and the induction hypothesis we have
\begin{align*}
\dfrac{F_{np}(x)}{1-x^{np}}&=\dfrac{F_{n}(x)}{1-x^n}- \dfrac{F_{n}(x^p)}{1-x^{np}}
=\sum_{m|n} \mu(m) \frac{x^m}{1-x^m} - \sum_{m|n} \mu(m) \frac{x^{pm}}{1-x^{pm}}\\
&=\sum_{m|n} \mu(m) \frac{x^m}{1-x^m} + \sum_{pm|pn} \mu(pm) \frac{x^{pm}}{1-x^{pm}}
=\sum_{m|np} \mu(m) \frac{x^m}{1-x^m}.
\end{align*}
\end{proof}
One consequence of \cref{prop:form_of_F_n} is a formula for $F_n'(-1)$. This will be useful later in \cref{sec:pq} to describe the value of the reduced Fekete polynomial $f_n(x)$ at $x = -1$.
\begin{prop}
\label{prop:derivative_at_minus1}
Let $n > 1$ be an integer and $n_0$ its radical. Then  
\[ F_n'(-1) = \sum_{1 \leq k \leq n, \gcd(k,n)=1} (-1)^{k-1} k = \begin{cases}
   \dfrac{n \varphi(n)}{2}  &  \text{if $n$ is even }\\
   \dfrac{ \mu(n_0) \varphi(n_0)}{2}&  \text{if $n$ is odd. }
\end{cases} \] 
\end{prop}
\begin{proof} 
Suppose $n$ is even. Then
\begin{align*}
F_n'(-1) = \sum_{\substack{1 \leq k \leq n\\ \gcd(k,n)=1}} (-1)^{k-1} k = \sum_{\substack{1 \leq k \leq n\\ \gcd(k,n)=1}} k = \sum_{\substack{1 \leq k \leq n/2\\ \gcd(k,n)=1}} \left(k \quad + \quad n-k \right) =\dfrac{n\varphi(n)}{2}.
\end{align*}
Now suppose $n$ is odd. By ~\cref{cor:derivative_negative_one}, we may assume that $n$ is squarefree. Differentiating \cref{eq:form_of_Fn}, we get
\[
F'_n(x)=nx^{n-1}\sum_{m|n} \mu(m) \frac{x^m}{x^m-1} +(x^n-1)\sum_{m|n} \mu(m) \frac{-mx^{m-1}}{(x^m-1)^2}.
\]
Hence
\[
F'_n(-1)=\dfrac{n}{2} \sum_{m|n} \mu(m) + \dfrac{1}{2} \sum_{m|n} \mu(m) m = \dfrac{1}{2} \mu(n)\sum_{m|n} \mu(n/m) m=\dfrac{1}{2}\mu(n)\varphi(n).
\]
\end{proof}

We now utilize \cref{prop:form_of_F_n} and the inclusion-exclusion principle to catch a big chunk of cyclotomic factors of $F_n$.
\begin{thm}
\label{thm:cyclotomic_factor_3}
    Let $n$ be a positive squarefree integer. Suppose $N > 1$ is a divisor of $n$ and $T$ is a partition of the set of all divisors of $N$ into two-element subsets $\{a_1, b_1\}, \{a_2, b_2\}, \ldots, \{a_l,b_l\}$. Let $D = \gcd_i (a_i + \mu(a_i)\mu(b_i) b_i)$. Then $\Phi_d$ divides $F_n$ for every $d|D$ such that $d \nmid n$.
\end{thm}
\begin{proof}
    Let $\zeta$ be a primitive $d$-th root of unity. Since $d \nmid n$, we can evaluate \cref{eq:form_of_Fn} at $\zeta$.
    \begin{align*}
        \frac{F_n(\zeta)}{1-\zeta^n} = \sum_{m|n} \frac{\mu(m)\zeta^m}{1-\zeta^m} = \sum_{m|(n/N)} \sum_{i=1}^l \frac{\mu(a_im)\zeta^{a_im}}{1-\zeta^{a_im}} + \frac{\mu(b_im)\zeta^{b_im}}{1-\zeta^{b_im}}.
    \end{align*}
    Without loss of generality, suppose that $\mu(a_i) = \mu(b_i)$ for $1 \leq i \leq l_0$ and $\mu(a_i) = -\mu(b_i)$ for $i > l_0$. Note that this means
    \begin{align*}
        2\sum\limits_{i=1}^{l_0} \mu(a_i) = \sum\limits_{i=1}^{l_0} \left( \mu(a_i) + \mu(b_i) \right) = \sum\limits_{i=1}^{l} \left( \mu(a_i) + \mu(b_i) \right) = \sum\limits_{k|N} \mu(k) = 0.
    \end{align*}
    For $i > l_0$, since $d | (a_i - b_i)$ we have $\zeta^{a_i} = \zeta^{b_i}$ and hence
    \begin{align*}
        \frac{\mu(a_im)\zeta^{a_im}}{1-\zeta^{a_im}} + \frac{\mu(b_im)\zeta^{b_im}}{1-\zeta^{b_im}} = 0.
    \end{align*}
    For $i \leq l_0$, since $d | (a_i + b_i)$ we have $\zeta^{a_i} = \zeta^{-b_i}$ and hence
    \begin{align*}
        \frac{\mu(a_im)\zeta^{a_im}}{1-\zeta^{a_im}} + \frac{\mu(b_im)\zeta^{b_im}}{1-\zeta^{b_im}} = -\mu(a_im).
    \end{align*}
    Putting things together, we get that
    \begin{align*}
        \frac{F_n(\zeta)}{1-\zeta^n} = \sum_{m|n} \frac{\mu(m)\zeta^m}{1-\zeta^m} = \sum_{m|(n/N)} \sum_{i=1}^{l_0} -\mu(a_im) = -\sum_{m|(n/N)} \mu(m) \sum_{i=1}^{l_0} \mu(a_i) = 0.
    \end{align*}
    This proves that $F_n(\zeta) = 0$, and hence $\Phi_d$ divides $F_n$.
\end{proof}

\begin{rem}
    \cref{thm:cyclotomic_factor_3} does not describe all cyclotomic factors of $F_N$. For instance, when $N = 105$, the cyclotomic polynomial $\Phi_{24}(x) = x^8 - x^4 + 1$ divides $F_N$, but this cannot be seen through \cref{thm:cyclotomic_factor_3}.
\end{rem}

We will use the following special case of this theorem the most.
\begin{cor}
\label{cor:d_dividing_pminus1}
    Let $r$ be a prime factor of $n$ and suppose $d$ is an integer dividing $r-1$ such that $d \nmid n$. Then $\Phi_d$ divides $F_n$. \qed
\end{cor}

For a prime number $p$, we also have a necessary condition for $\Phi_p$ to divide $F_n$.
\begin{prop}
\label{prop:prime_divisor_2}
    Let $p$ be a prime number such that $p \nmid n$. If $\Phi_p$ is a factor of $F_n$ then $p|\varphi(n)$, i.e., there exists a prime divisor $r$ of $n$ such that $p|r-1.$
\end{prop}
\begin{proof}
    If $\Phi_p$ is a factor of $F_n$ then $F_n=\Phi_pQ$ for some $Q\in \Z[x]$. Hence $\varphi(n)=F_n(1)=\Phi_p(1)Q(1)=pQ(1)$ and thus $p\mid \varphi(n)$.\\
    Alternatively, by \cref{rem:equidistribute_prime_d}, we know that \cref{conj:equidistribution} holds in this case. So the elements of $\{1 \leq a \leq n | \gcd(a,n) = 1\}$ when reduced modulo $p$, equidistribute in $S = \{0,1,\ldots,p-1\}$. Therefore $p = \#S$ divides $\varphi(n)$.
\end{proof}

Combining \cref{cor:d_dividing_pminus1} and \cref{prop:prime_divisor_2}, we have the following. 
\begin{cor}
\label{cor:prime_divisor}
    Let $p$ be a prime number. Then $\Phi_p$ is a factor of $F_n$ if and only if $p \nmid n$ and there exists a prime divisor $r$ of $n$ such that $p|r-1$. \qed
\end{cor}

We observe that all cyclotomic factors described by \cref{thm:cyclotomic_factor_3} satisfy the equidistribution condition mentioned in \cref{prop:equidistribution_condition_new}.

\begin{thm}
\label{thm:equidistribution_provedcases}
Let $n, N, D$ be as in \cref{thm:cyclotomic_factor_3}. Suppose $D \nmid n$. Let $d = D$ and $d_1 = \gcd(d,n)$. Then the elements of $\{1\leq a\leq n\mid \gcd(a,n)=1\}$ when reduced modulo $d$, equidistribute in $S = \{ 1 \leq a \leq d | \gcd(a,d_1) = 1\}$. Equivalently, we have that $x^d-1$ divides $F_n - c F_S = F_n - c \sum\limits_{s \in S} x^s$ for some constant $c$.
\end{thm}
\begin{proof}
    Let $c = \frac{\varphi(n) d_1}{d \varphi(d_1)}$ and consider the polynomial
    $W(x) = F_n(x) - c F_S(x)$. By \cref{rem:equivalence_equidistribution} it is enough to show that $x^d-1$ divides $W(x)$.
    
    Let $m$ be any integer dividing $d$, and $\zeta_m$ be a primitive $m$-th root of unity. If $m \nmid n$ then we know $F_n(\zeta_m) = 0$ by \cref{thm:cyclotomic_factor_3} and $F_S(\zeta_m) = 0$ by \cref{lem:F_S}. Hence $W(\zeta_m) = 0$.
    If $m | n$ then $m | d_1 = \gcd(d,n)$, and we again deduce that $W(\zeta_m) = 0$ by \cref{prop:value_at_zeta_d} and \cref{lem:F_S}.
\end{proof}

\subsection{Multiplicity of cyclotomic factors}
In this section, we determine the multiplicity of cyclotomic factors of $F_n.$ To do so, we need to investigate the value $F'_n(\zeta_d)$ whenever $F_n(\zeta_d) = 0$. By \cref{cor:d_divisor_n}, $\zeta_d$ is not a root of $F_n$ if $d|n$. Therefore, we can assume that $d \nmid n$ in our investigation. By \cref{prop:form_of_F_n} we have 
\begin{equation} \label{eq:derivative_of_f}
F'_n(x)=nx^{n-1}\sum_{m|n} \mu(m) \frac{x^m}{x^m-1} +(x^n-1)\sum_{m|n} \mu(m) \frac{-mx^{m-1}}{(x^m-1)^2}.
\end{equation}
We introduce the following polynomial
\[ G_n (x) = (x^n-1)^2 \sum_{m|n} \mu(m) \frac{mx^{m}}{(x^m-1)^2}, \]
so that \cref{eq:derivative_of_f} becomes
\begin{equation}
\label{eq:f_and_g}
x(x^n-1)F'_n(x)=nx^{n}F_n(x)-G_n(x).
\end{equation}
Expanding this out using the definition of $F_n$, we have the following explicit formula
\begin{equation}
\begin{split}
    \label{eq:G_n}
    G_n(x) & = nx^n F_n(x)-(x^{n+1}-x) F'_n(x) = \sum_{\substack{1 \leq a \leq n \\ \gcd(a,n)=1}} \left(ax^a + (n-a)x^{n+a}\right)\\
    & = \sum_{\substack{1 \leq a \leq n \\ \gcd(a,n)=1}} (n-a)[x^{n+a}+x^{n-a}] = x^n \sum_{\substack{1 \leq a \leq n \\ \gcd(a,n)=1}} (n-a)[x^a + x^{-a}]
\end{split}
\end{equation}

\begin{lem} \label{lem:big_d}

Let $d$ be a positive integer and $n$ a positive squarefree integer. Let $\zeta = \zeta_d$ be a primitive $d$-th root of unity. 
\begin{enumerate}
    \item  If $n$ is even then $\zeta_4$ is a  root of $G_n.$
    \item If $d \geq 2n$ then $G_n(\zeta) \neq 0$ unless $n=2$ and $d=4.$
    \end{enumerate}
\end{lem}
\begin{proof}
    The first part follows immediately from \cref{eq:G_n} and the fact that $\zeta_4^a + \zeta_4^{-a} = 0$ for odd integers $a$. 
    For the second part we first rewrite \cref{eq:G_n} as follows
    \begin{align*} 
    x^{-n} G_n(x) &= \sum_{\substack{1 \leq a \leq \frac{n}{2}\\ \gcd(a,n)=1}} \left[(n-a)(x^a+x^{-a})+a (x^{n-a}+x^{a-n})  \right] \\ 
                  &= \sum_{\substack{1 \leq a \leq \frac{n}{2}\\ \gcd(a,n)=1}} (n-2a)[x^a+x^{-a}] + \sum_{\substack{1 \leq a \leq \frac{n}{2}\\ \gcd(a,n)=1}} a [x^a+x^{-a}+x^{n-a}+x^{a-n}].
    \end{align*}
    Working over $\C$, if we let $\zeta_d = e^{\frac{2 \pi i}{d}}$, then by Euler formula we have 
    \[ \zeta_d^{-n} G_n(\zeta_d) = 2\sum_{\substack{1 \leq a \leq \frac{n}{2}\\ \gcd(a,n)=1}} (n-2a) \cos\left(\frac{2 \pi a}{d}\right) + 4a\sum_{\substack{1 \leq a \leq \frac{n}{2}\\ \gcd(a,n)=1}} \cos\left(\frac{\pi n}{d}\right) \cos\left(\frac{(n-2a) \pi}{d}\right) \]
    Since $d \geq 2n$ and $1 \leq a \leq \frac{n}{2}$, each term in the sum is non-negative and we get
    $\zeta_d^{-n} G_n(\zeta_d) \geq 0$.
    Further, if $G_n(\zeta_d)=0$ then it must be the case that $(n-2)\cos\left(\frac{2 \pi}{d}\right)=0$ and $\cos\left(\frac{n\pi }{d}\right)\cos\left(\frac{(n-2)\pi}{d}\right) = 0$. This implies that $d=4$ and $n=2.$   
\end{proof}

We remark that by \cref{eq:f_and_g}, $\zeta = \zeta_d$ is simple root of $F_n$ if and only if  $F_n(\zeta)=0$ and $G_n(\zeta) \neq 0$. 
Using the definition of $G_n$, we also get the following recursive formula for $G_n$ analogous to the one for $F_n$ described in \cref{prop:recursive_F}. If $p \nmid n$, then
\[ G_{np}(x) = \left(\frac{1-x^{np}}{1-x^n} \right)^2 G_n(x) - pG_n(x^p).\]
To study the multiplicity of $\Phi_d$ as a factor of $G_n$, we will use a general proposition.
\begin{prop} \label{prop:multiplicity_nonsense}
    Let $G(x)=\frac{P(x)}{Q(x)}$ where $P(x), Q(x)$ are polynomials with rational coefficients. 
    Let $d$ be a positive number and $\zeta = \zeta_d$ a primitive $d$-th root of unity such that $Q(\zeta) \neq 0$. Let $p$ be a prime number such that $p \nmid d$.
    Let $G_p(x)$ be the following rational function 
    \[ G_p(x) = G(x) - pG(x^p). \] Then  
    \[ \mult_{\zeta_d}(G) = \mult_{\zeta_d}(G_p).\]
\end{prop}

\begin{proof}

By induction, we can see that for each $k \geq 0$ 
\begin{equation} \label{eq:recursive_G} G_p^{(k)}(x) = G^{(k)}(x) - p^{k+1} x^{k(p-1)} G^{(k)}(x^p) + \sum_{0 \leq h \leq k-1} M_{k,h}(x) G^{(h)}(x^p), 
\end{equation} 
where $M_{k,h} \in \Q[x].$ In order to prove the above statement, we will show that for all $k \geq 0$, $G(\zeta) = \cdots= G^{(k)}(\zeta)=0$ if and only if $G_p(\zeta) = \cdots= G_p^{(k)}(\zeta)=0$.

Let us consider the base case $k=0.$ Suppose that $G(\zeta)=0.$ Since $\zeta$ and $\zeta^p$ are Galois-conjugate and $G$ is a rational function with rational coefficients, $G(\zeta^p)=0$ as well. We conclude that $G_p(\zeta)=0$. Conversely, suppose that $G_p(\zeta)= 0$, i.e., $G(\zeta)  = pG(\zeta^p)$. Since $G_p \in \Q(x)$, we have this equality for any primitive $d$-th root of unity. In particular $G(\zeta) = p G(\zeta^p) = p^2 G(\zeta^{p^2}) = \ldots = p^{\varphi(d)} G(\zeta)$. This means that $G(\zeta) = 0$.

Assume that the statement holds for $k-1.$ Let us show that it is true for $k$ as well. First, suppose that 
    $ G(\zeta) = \cdots= G^{(k)}(\zeta)=0$.
By the induction hypothesis, we already know that 
    $ G_p(\zeta) = \cdots= G_p^{(k-1)}(\zeta)=0$.
Since $\zeta$ and $\zeta^p$ are Galois conjugate over $\Q$, we know that $G^{(h)}(\zeta^p)=0$ for $0 \leq h \leq k$ as well. By \cref{eq:recursive_G}, we conclude that $G_p^{(k)}(\zeta)=0$. Conversely, assume that 
$ G_p(\zeta) = \cdots= G_p^{(k)}(\zeta)=0$.
By the induction hypothesis, we know that 
$ G(\zeta) = \cdots= G^{(k-1)}(\zeta)=0$.
Let us show that $G^{(k)}(\zeta)=0$. Because of Galois conjugacy we again know that 
$ G(\zeta^p) = \cdots =G^{(k-1)}(\zeta^p)=0$.
Additionally, since $G_{p}^{(k)}(\zeta)=0$, \cref{eq:recursive_G} tells us that 
\[ G^{(k)}(\zeta) = p^{k+1} \zeta^{k(p-1)} G^{(k)}(\zeta^p) .\]
Since this equality holds for any primitive $d$-th root of unity, a similar argument as in the base case lets us conclude that $G^{(k)}(\zeta)=0$.
\end{proof}

\begin{cor}
\label{cor:relation_Gn}
Suppose that $p\nmid nd$ and $d \nmid n$. Let $\zeta = \zeta_d$ be a primitive $d$-th root of unity. Then 
$\mult_{\zeta}(G_n) = \mult_{\zeta}(G_{np})$.
\end{cor}
\begin{proof} 
Consider the slightly modified function
\[ \tilde{G}_n(x) = \frac{G_n(x)}{(1-x^n)^2}. \]
Since $d \nmid np$, $\zeta=\zeta_{d}$ is not a root of $x^{np}-1$. So $\mult_{\zeta}(G_n) = \mult_{\zeta}(\tilde{G}_n)$ and $\mult_{\zeta}(G_{np}) = \mult_{\zeta}(\tilde{G}_{np})$. Furthermore we have 
$\tilde{G}_{np}(x) = \tilde{G}_n(x) - p \tilde{G}(x^p)$.
Applying \cref{prop:multiplicity_nonsense} to $\tilde{G}_n$ gives the desired statement. 
\end{proof}

When $n=q$ is a prime number, we have 
\[ G_q (x) = (x^q-1)^2\left[\frac{x}{(1-x)^2} - \frac{qx^{q}}{(1-x^q)^2}\right] = xH_q(x), \]
where 
\[ H_q(x) = \left(\frac{1-x^q}{1-x} \right)^2 - qx^{q-1}.\]
Note that $H_2(x)=x^2+1=\Phi_4(x)$.
\begin{lem}
\label{lem:G_q}
    The polynomial $H_q(x+1)$ is an Eisenstein polynomial at the  prime $q.$ Consequently, $H_q(x)$ is irreducible. Moreover, if $q$ is odd, then $\zeta_d$ is not a root of $H_q$ for any $d$, and hence $\zeta_d$ is not a root of $G_q$.
\end{lem}

\begin{proof}
Over $\F_q$ we have 
$H_q(x) = (1-x)^{2(q-1)}$.
Additionally 
$H_q(1) = q^2-q$.
This shows that $H_q(x+1)$ is an Eisenstein polynomial at the prime $q$.

Suppose on the contrary that $\zeta_d$ is a root of $H_q.$ Since $\Phi_d(x)$ and $H_q(x)$ are both irreducible, we must have $\Phi_d=H_{q}.$ In particular 
\[ q^2- q = H_q(1) = \Phi_d(1).\]
If $d=p^k$ is a prime power  ($k\geq 1$) then $\Phi_d(1)=p$. In this case, we have $p=q^2-q$, which implies that $q|p$ and hence $q=p=2$. If $d$ is not a prime power then $\Phi_d(1)=1$. But there is no prime $q$ such that $q^2-q=1$.
\end{proof}
With these preparations, we can now describe the multiplicity of any cyclotomic factor $\Phi_d$ of $F_n$.
\begin{thm}
\label{thm:multiplicity_of_Phid}
Let $n>1$ be a squarefree integer. Let $d$ be a positive integer and $\zeta = \zeta_d$ a primitive $d$-th root of unity such that $F_n(\zeta)=0$. Then we have the following. 
\begin{enumerate}
    \item If $d \neq 4$, then $\zeta_d$ is a simple roof of $F_n.$
    \item If $d = 4$ and $n$ is odd,  then $\zeta_d$ is a simple roof of $F_n.$
    \item If $d = 4$ and $n$ is even,  then $\zeta_d$ is a double root of $F_n.$
\end{enumerate}
\end{thm}
\begin{proof} 
We will prove the first and second statements by showing the contrapositive. Suppose that $\zeta = \zeta_d$ is a repeated root of $F_n.$ Then we know that $d \nmid n$ and that $F_n(\zeta)= G_n(\zeta)=0.$ Let $d_1 = \gcd(n,d).$ Let us write $n= d_1 n_1$ where $\gcd(n_1, d)=1.$ Since $n$ is a squarefree integer, this also implies that $\gcd(n_1, d_1 d)=1.$ Since $G_n(\zeta) = G_{d_1 n_1}(\zeta)=0$, \cref{cor:relation_Gn} implies that $G_{d_1}(\zeta)=0$ as well. Since $d \geq 2d_1$ we conclude by \cref{lem:big_d} that $d=4$ and $d_1=2$ meaning that $n$ is even.

When $d=4$ and $n$ is even, we know by \cref{cor:relation_Gn} that 
$\mult_{\zeta_4}(G_{n}) = \mult_{\zeta_4}(G_2) =1$.
\cref{eq:f_and_g} then shows that $F_n(\zeta_4) = F_{n}'(\zeta_4)=0$ but $F_{n}''(\zeta_4) \neq 0.$ This proves that $\zeta_4$ is a double root of $F_n$.
\end{proof}
\begin{rem}
It is not true that $\mult_{\zeta_4}(F_n)=2$ for all even squarefree integers $n$. For instance, \cref{expl:2p} says that for odd primes $p$, $F_{2p}(\zeta_4) = 0$ if and only if $p \equiv 1 \pmod{4}$. \cref{thm:multiplicity_of_Phid} only asserts $\mult_{\zeta_4}(F_n)=2$ on the condition that $\zeta_4$ is a root of $F_n$. In the next section we will classify all $n$ such that $\zeta_4$ is a root of $F_n$. 
\end{rem}

\subsection{Cyclotomic factors of $F_n$ with small degree}
\label{subsec:small_deg_factors}

In this section, we give explicit descriptions of the squarefree integers $n$ such that $\Phi_d$ is a factor of $F_n$, for each $d \leq 7$. The cases $d=2, 3, 5, 7$ are already dealt with by \cref{cor:prime_divisor}. We will show that the same statement also holds for $d = 4, 6$.
\begin{rem}
This statement fails for $d = 8$, with the smallest example coming from $n = 15$. By \cref{thm:cyclotomic_factor_3}, we know that $\Phi_8$ is a factor of $F_{n}$, but $n$ has no prime factor congruent to $1 \pmod 8$.
\end{rem}

In order to simplify our calculations, we consider the modified function $\tilde{F}_n(x) = \frac{F_n(x)}{1-x^n}$.
The value $\tilde{F}_n(\zeta_d)$ is well-defined as long as $d \nmid n$, and furthermore it is zero if and only if $F_n(\zeta_d)=0.$ We also recall the recursive formula $\tilde{F}_{np}(x) = \tilde{F}_n(x) - \tilde{F}_{n}(x^p)$ for $p \nmid n$ from \cref{prop:recursive_F}, which will be used repeatedly in this section. 

We now consider the case $d=4$. By \cref{cor:d_dividing_pminus1}, if $n$ has a prime factor $p \equiv 1 \pmod{4}$ then $F_n(\zeta_4)=0.$ It turns out that the converse is true as well.

\begin{prop}
Suppose that $n=2^s p_1 p_2 \ldots p_r$ where $s \in \{0,1 \}$ and $p_1, \ldots, p_r$ are distint odd primes of the form $4k+3$. Then 
$ \tilde{F}_{n} (\zeta_4) = 2^{r-1} \zeta_4$.
In particular, $F_n(\zeta_4) \neq 0.$
\end{prop}

\begin{proof}
Let us write $\zeta=\zeta_4$ for simplicity. If $s=1$ and $m = p_1 p_2 \ldots p_r$ we get using the recursive formula that
\[ \tilde{F}_{n}(\zeta) = \tilde{F}_{m}(\zeta) - \tilde{F}_m(\zeta^2) = \tilde{F}_m(\zeta)-\tilde{F}_m(-1) = \tilde{F}_m(\zeta),\]
since $F_m(-1)=0$ by \cref{cor:d_dividing_pminus1}. So it is sufficient to prove the statement for odd $n$. We will do this by induction on the number of prime factors of $n$.

When $n = p$ is a prime of the form $4k+3$, we have $\tilde{F_n}(\zeta) 
= \frac{1}{1-\zeta^p} \frac{\zeta(1- \zeta^{p-1})}{1-\zeta} = \zeta$. Suppose that the formula is true when $n$ is a product of $r$ primes of the form $4k+3$. Let us now suppose that $n=p_1 p_2 \ldots p_{r+1}$. Let $m=p_2  \ldots p_{r+1}$. By the induction hypothesis, we know that $\tilde{F}_m(\zeta)=2^{r-1} \zeta$. By Galois conjugation, we get that $\tilde{F}_m(\zeta^3) = 2^{r-1}\zeta^3 = -2^{r-1}\zeta$. Now applying the recursive formula gives us the desired statement.  
\[ \tilde{F}_n(\zeta) = \tilde{F}_m(\zeta)- \tilde{F}_{m}(\zeta^{p_1}) =\tilde{F}_m(\zeta)- \tilde{F}_{m}(\zeta^{3}) = 2^r \zeta. \]
This finishes the proof by induction.
\end{proof}

For $d=6$, \cref{cor:d_dividing_pminus1} tells us that if $6 \nmid n$ and $n$ has a prime factor $p \equiv 1 \pmod{6}$, then $F_n(\zeta_6)=0$. It again turns out that the converse is true. We first need a lemma.

\begin{lem} \label{lem:root_zeta_3}
    Suppose that $n=p_1 p_2 \ldots p_r$ for distinct primes $p_1, \ldots, p_r$ of the form $3k+2$. Then $\tilde{F}_{n} (\zeta_3) = 2^{r-1} (1+2\zeta_3)/3$.
\end{lem}
\begin{proof}
When $r=1$, direct calculation gives
$\tilde{F}_{p_1}(\zeta_3) =\frac{1}{1-\zeta_3^{p_1}} \frac{\zeta_3 (1-\zeta_3^{p_1-1})}{1-\zeta_3} = (1+2\zeta_3)/3$.
In the general case, we let $m=p_2 \ldots p_r$ and use the recursive formula 
\[ \tilde{F}_n(\zeta_3)= \tilde{F}_{mp_1}(\zeta_3) = \tilde{F}_m(\zeta_3) - \tilde{F}_m(\zeta_3^{p_1}) = \tilde{F}_m(\zeta_3) - \tilde{F}_m(\zeta_3^2).\] 
Now the induction hypothesis $\tilde{F}_m(\zeta_3) = 2^{r-2}(1+2\zeta_3)/3$ and Galois conjugacy implies that $\tilde{F}_n(\zeta_3) = \frac{2^{r-2}}{3} \left((1+2\zeta_3) - (1+2\zeta_3^2) \right) = 2^{r-1} (1+2\zeta_3)/3$.
\end{proof}

\begin{prop}
    $F_n(\zeta_6)=0$ if and only if $6 \nmid n$ and there exists a prime divisor $p$ of $n$ such that $p \equiv 1 \pmod{6}.$
\end{prop}

\begin{proof}
The ``if'' part follows from \cref{cor:d_dividing_pminus1}. So we focus on the ``only if'' part. By \cref{cor:d_divisor_n} we know that if $6 \mid n$ then $F_n(\zeta_6) \neq 0$. So let us assume that $6 \nmid n$. Suppose that $n$ has no prime factor $p \equiv 1 \pmod{6}$. We consider three cases according to whether $\gcd(n,6)=1$, $2$ or $3$.

In the first case, $n = p_1 p_2 \ldots p_r$ for distinct primes $p_i \equiv 5 \pmod 6$, and we can show by induction that 
$\tilde{F}_n(\zeta_6) = 2^{r-1} (2\zeta_6-1) \neq 0$.

In the second case, we write $n=2m = 2p_1 p_2 \ldots p_r$ for distinct primes $p_i \equiv 5 \pmod 6$. The recursive formula then tells us that
\[ \tilde{F}_n(\zeta_6) = \tilde{F}_m(\zeta_6) - \tilde{F}_m(\zeta_6^2) = 2^{r-1} (2\zeta_6 - 1) - 2^{r-1} \frac{1+2\zeta_6^2}{3} \neq 0, \]
where we have used the formula from the first case, and \cref{lem:root_zeta_3}.

In the last case, we write $n=3m = 3p_1 p_2 \ldots p_r$ for distinct primes $p_i \equiv 5 \pmod 6$ and get
\[ \tilde{F}_n(\zeta_6) = \tilde{F}_m(\zeta_6) - \tilde{F}_m(\zeta_6^3) =  \tilde{F}_m(\zeta_6) - \tilde{F}_m(-1) = \tilde{F}_m(\zeta_6) = 2^{r-1} (2\zeta_6 - 1) \neq 0. \]
We conclude that in all cases $\tilde{F}_n(\zeta_6) \neq 0.$ This completes the proof. 
\end{proof}


\section{The case $n = pq$}
\label{sec:pq}

In this section, we study more closely the case where $n$ is a product of two odd primes $p > q$. The following proposition is a direct consequence of \cref{thm:cyclotomic_factor_3}. If we assume that \cref{conj:equidistribution} holds, this proposition in fact captures \emph{all} cyclotomic factors of $F_{pq}$.
\begin{prop}\label{prop:cyclotomic_factors_pq}
Let $n$ be a product of two odd primes $p > q$. Suppose $d > 1$ and one of the following holds:
\begin{enumerate}[(a)]
    \item $d$ divides $q-1$;
    \item $d$ divides $p-1$ and $d \neq q$;
    \item $d$ divides $\gcd(qp+1,p+q)$. \qed
\end{enumerate}
Then the $d$-th cyclotomic polynomial $\Phi_d$ divides $F_n$.
\end{prop}

\begin{thm}
\label{thm:cyclotomic_factors_pq1}
   Assume \cref{conj:equidistribution} is true. Then \cref{prop:cyclotomic_factors_pq} is a complete characterization of all cyclotomic factors of $F_n$, when $n = pq$ is a product of two odd primes $p > q$.
\end{thm}
\begin{proof}
    We have
    \begin{align}\label{eqn:Fpq_expression}
    \begin{split}
    (x^p-1)(x^q-1)F_n(x) =(x^p-1)\sum_{1 \leq i \leq q-1}\left(x^{qp+i}+x^{ip}-x^i-x^{ip+q}\right)\\
    =x^{qp} - x^p - x^{qp+q} + x^{p+q} +
    \sum_{1 \leq i \leq q-1} \left(x^{(q+1)p+i}-x^{qp+i}-x^{p+i}+x^i\right).
    \end{split}
    \end{align}
    To evaluate a polynomial at a primitive $d$-th root of unity $\zeta_d$, it is enough to consider its reduction modulo $x^d-1$. In particular, we may reduce the monomial exponents modulo $d$.
    
    By \cref{cor:d_divisor_n} we know that $d \nmid n$. Suppose $d = d_1 e_1$ with $d_1 = \gcd(d,n)$ and $e_1 > 1$. We have the following cases depending on whether $d_1$ is $p, q$ or $1$.\\
    \textbf{Case 1:} Suppose $d_1 = p$. The equidistribution from \cref{conj:equidistribution} implies that $\#S = \varphi(d_1)e_1 = \varphi(p)e_1$ divides $\varphi(n)=\varphi(p)\varphi(q)$. Hence $e_1 | \varphi(q)$ and $d | p\varphi(q)$. Reducing \cref{eqn:Fpq_expression} modulo $x^{p\varphi(q)}-1$ without loss of generality, we get
    \begin{align*}
        \sum_{1 \leq i \leq q-1} \left(x^{2p+i}-2x^{p+i}+x^i\right) = (x^p-1)^2\sum_{1 \leq i \leq q-1} x^i = (x^p-1)^2 \frac{x(x^{q-1}-1)}{x-1},
    \end{align*}
    whose value at $\zeta_d$ is clearly non-zero.\\
    \textbf{Case 2:} Suppose $d_1 = q$. \cref{conj:equidistribution} implies that $\#S = \varphi(d_1)e_1$ divides $\varphi(n)$. Hence $e_1 | \varphi(p)$ and $d | q\varphi(p)$. Reducing \cref{eqn:Fpq_expression} modulo $x^{q\varphi(p)}-1$, we get
    \begin{align*}
        \sum_{1 \leq i \leq q} \left(x^{p+q-1+i}-x^{q+i}-x^{p-1+i}+x^i\right) &= (x^{p+q-1}-x^q-x^{p-1}+1) \frac{x(x^q-1)}{x-1}\\
        &= (x^{p-1}-1)\frac{x(x^q-1)^2}{x-1},
    \end{align*}
    whose value at $\zeta_d$ is zero if and only $d | p-1$ and $d \neq q$.\\
    \textbf{Case 3:} Suppose $d_1 = 1$. \cref{conj:equidistribution} implies that $\#S = e_1 = d$ divides $\varphi(n)$. Reducing \cref{eqn:Fpq_expression} modulo $x^{\varphi(n)}-1$, we get
    \begin{align*}
        & x^{p+q-1} - x^p - x^{p+2q-1} + x^{p+q} + \sum_{1 \leq i \leq q-1} \left(x^{2p+q+i-1}-x^{p+q+i-1}-x^{p+i}+x^i\right)\\
        &= \sum_{i \in S_1} x^i - \sum_{i \in S_2} x^i,
    \end{align*}
    where $S_1 = \{ 1 \leq i \leq q-1 \} \cup \{p+q-1,p+q\} \cup \{2p+q \leq i \leq 2p+2q-2\}$ and $S_2 = \{p \leq i \leq p+2q-1\}$. Note that $S_2$ is a sequence of $2q$ consecutive integers. Therefore, this polynomial evaluated at $\zeta_d$ is zero if and only if $S_1$ modulo $d$ is equal to the same set of consecutive residues modulo $d$. This happens if and only if $d$ divides $p-1$, $q-1$, or $p+q$. If $d$ divides $p+q$, then $d$ also divides $pq+1$ because $pq+1 = \varphi(n)+(p+q)$.
\end{proof}

Let $S_n$ be the set of integers $d$ described in \cref{prop:cyclotomic_factors_pq}, namely 
\begin{equation}
\label{eq:S_n}
S_n  = \{d>1, d\not=q, d \mid p-1 \} \cup \{d>1,  d \mid q-1 \} \cup \{d>1, d\mid \gcd(pq+1,p+q) \}.
\end{equation}
\begin{definition}
\label{defn:fn}
Suppose $n = pq$ for odd primes $p > q$. Let $S_n$ be as above. We define the Fekete polynomial $f_n(x)\in \Z[x]$ to be the polynomial such that
\begin{align*}
    F_n(x)&=f_n(x) \cdot x \cdot \prod_{d\in S_n}\Phi_d(x)
\end{align*}
\end{definition}

\begin{prop}
\label{prop:values_at_1and-1}
Suppose $n = pq$ for odd primes $p> q$. Let $f_n$ be the Fekete polynomial as in \cref{defn:fn}. Let $D_1 = \gcd(p-1,q-1)$, $D_2 = \gcd(pq+1,p+q)$, $D_3 = \gcd(pq+1,p+q,p-1) = \gcd(p-1,q+1)$, $D_4 = \gcd(pq+1,p+q,q-1) = \gcd(p+1,q-1)$. Then $f_n$ is a reciprocal polynomial of even degree. More precisely,
\begin{align*}
    \deg(f_n) = \begin{cases}
    pq-p-q-1+D_1+D_3+D_4-D_2 \quad &\text{if } p \not\equiv 1 \pmod{q}\\
    pq-p-2+D_1+D_3+D_4-D_2 \quad &\text{if } p \equiv 1 \pmod{q}.
    \end{cases}
\end{align*}
Furthermore, we have
\begin{align*}
    &f_n(1) = \begin{cases}
    \dfrac{D_1D_3D_4}{2D_2} \quad &\text{if } p \not\equiv 1 \pmod{q}\\
    \dfrac{qD_1D_3D_4}{2D_2} \quad &\text{if } p \equiv 1 \pmod{q},\\
    \end{cases} \\
    &f_n(-1) = \frac{-D_1D_3D_4}{2D_2}.
\end{align*}
\end{prop}
\begin{proof}
Let
\begin{align*}
    f(x) = \prod_{\substack{d \mid q-1 \\ d \neq 1}} \Phi_d(x), \quad g(x) = \prod_{\substack{d \mid p-1 \\ d \neq q \\ d \nmid q-1}} \Phi_d(x), \quad h(x) = \prod_{\substack{d \mid \gcd(pq+1,p+q) \\ d \nmid q-1, d \nmid p-1}} \Phi_d(x).
\end{align*}
Then we have $F_n(x) = xf(x)g(x)h(x)f_n(x)$. Using the inclusion-exclusion principle, we get the following formulas for $f, g, h$ in this decomposition:
\begin{align*}
    &f(x) = \frac{1-x^{q-1}}{1-x} = \frac{F_q(x)}{x}, \quad g(x) = \begin{cases}
    \frac{1-x^{p-1}}{1-x^{D_1}} \quad &\text{if } p \not\equiv 1 \pmod{q}\\
    \frac{(1-x^{p-1})(1-x)}{(1-x^{D_1})(1-x^q)} \quad &\text{if } p \equiv 1 \pmod{q},\\
    \end{cases}\\
    &h(x) = \frac{(1-x^{D_2})(1-x^2)}{(1-x^{D_3})(1-x^{D_4})}.
\end{align*}
Since each of the polynomials $F_n(x)/x$, $f, g, h$ are reciprocal, we deduce that $f_n$ is also a reciprocal polynomial. The formula for $\deg(f_n)$ can also be deduced easily from the explicit descriptions of $f,g,h$ given above.

It is also clear from this description that
\begin{align*}
    & f(1) = q-1, \quad g(1) = \begin{cases} \frac{p-1}{D_1} \quad \text{if } p \not\equiv 1 \pmod{q}\\ \frac{p-1}{qD_1} \quad \text{if } p \equiv 1 \pmod{q}, \end{cases}\\
    & h(1) = \frac{2D_2}{D_3D_4}.
\end{align*}
Since $F_n(1) = (p-1)(q-1)$, we infer the value of $f_n(1)$.

Note that $D_i$ is even for $1 \leq i \leq 4$, and hence $g(-1), h(-1) \neq 0$ whereas $F_n(-1) = f(-1) = 0$. 
Hence 
$F'_n(-1)=(-1)f'(-1)g(-1)h(-1)f_n(-1)$.
Thus, we calculate $F_n'(-1)$ and $f'(-1)$ using \cref{prop:derivative_at_minus1}, and $g(-1)$ and $h(-1)$ using calculus, to infer the value of $f_n(-1)$:
\begin{align*}
    F_n'(-1) = \frac{(p-1)(q-1)}{2}, \quad &f'(-1) = -F_q'(-1) = \frac{q-1}{2},\\
    g(-1) = \frac{p-1}{D_1}, \quad &h(-1) = \frac{2D_2}{D_3D_4}.
\end{align*}
\end{proof}

\begin{cor}
$f_{pq}(x)$ is not a product of cyclotomic polynomials. In particular, $f_{pq}(x)$ is not a cyclotomic polynomial. 
\end{cor}
\begin{proof}
Suppose on the contrary that  
$f_{pq}(x) = \prod_{i=1}^r \Phi_{m_i}(x)$,
where $1 \leq m_1 \leq m_2 \dots \leq m_r$ are positive integers. Since $f_{pq}(1)f_{pq}(-1) \neq 0$, we can assume that $m_1 >2$. By \cite[Lemma 7]{bzdega2016cyclotomic}, we have $\Phi_{m_i}(-1)>0$ for all $1 \leq i \leq r.$ Consequently $f_{pq}(-1)>0.$ This contradicts the above determination of $f_{pq}(-1).$
\end{proof}

\begin{definition}
\label{def:g}
    We define $g_n$ to be the trace polynomial of $f_n$, i.e., it is the unique polynomial such that $g_n \left(x+\frac{1}{x}\right) = x^{-\deg(f_n)/2}f_n(x)$.
\end{definition}
\begin{prop} \label{prop:discriminant_f}
Suppose $n = pq$ for odd primes $p > q$. Let $f_n$ be the Fekete polynomial as in \cref{defn:fn}, and $g_n$ its trace polynomial as in \cref{def:g}. Assume $\disc(g_n)$ (or equivalently $\disc(f_n)$) is nonzero).\\
If $p \not\equiv 1 \pmod{q}$, then up to squares, we have
\begin{align*}
    \disc(f_n) = \begin{cases}
    -1 \quad &\text{if } p,q \equiv 1 \pmod{4}\\
    1 \quad &\text{otherwise}.
    \end{cases}
\end{align*}
If $p \equiv 1 \pmod{q}$, then up to squares, we have
\begin{align*}
    \disc(f_n) = \begin{cases}
    q \quad &\text{if } p \equiv 3 \pmod{4}~\text{ and } q \equiv 1 \pmod{4}\\
    -q \quad &\text{otherwise}.
    \end{cases}
\end{align*}
\end{prop}
\begin{proof}
Since $f_n$ is a reciprocal polynomial,
\begin{align*}
    \disc(f_n)=(-1)^{\deg(f_n)/2}f_n(1)f_n(-1) \disc(g_n)^2.
\end{align*}
Therefore \cref{prop:values_at_1and-1} tells us that up to squares, we have
\begin{align*}
    \disc(f_n) = \begin{cases}
    (-1)^{\deg(f_n)/2}(-1) \quad &\text{if } p \not\equiv 1 \pmod{q}\\
    (-1)^{\deg(f_n)/2}(-q) \quad &\text{if } p \equiv 1 \pmod{q}.
    \end{cases}
\end{align*}
Calculating $\deg(f_n)$ modulo $4$, again using \cref{prop:values_at_1and-1}, we obtain the stated result.
\end{proof}

We now investigate roots of the Fekete polynomial $F_{pq}$ in $\overline{\F}_p$. The objective of this study is to prove that $f_n$ is separable over $\Z$. We will see later that, in some special cases, this can be done by proving separability over $\F_p$. We first recall the following definition.

\begin{definition}
    Let $f,g$ be two polynomials. The Wronskian  $W(f,g)$ of $f$ and $g$ is defined by the formula 
    $W(f,g)=f'g-g'f$. 
\end{definition}

Let $q$ be a prime. We now introduce the polynomial
$u_q(x) = W(s_q(x),F_q(x)) \in \Z[x]$
where $s_q(x) = x^q-1$ and $F_q(x) = \sum\limits_{i=1}^{q-1} x^i$. Then $u_q$ has the following explicit formula 
\begin{equation}
\label{def:uq}
    u_q(x) = \sum_{1 \leq i \leq q-1} (q-i)x^{q-1+i} + \sum_{1 \leq i \leq q-1} ix^{i-1} = \Phi_q(x)^2 - q x^{q-1}.    
\end{equation}

\begin{lem} 
\label{lem:u_mod_q}
Over $\F_q$, we have $u_q(x)= (x-1)^{2q-2} \pmod q$.
\end{lem}

\begin{proof}
We have 
\[ F_q(x)= x \frac{x^{q-1}-1}{x-1} = \frac{x^q-x}{x-1}, \quad F'_q(x) = \frac{(qx^{q-1}-1)(x-1)-(x^q-x)}{(x-1)^2}.\]
Over $\F_q$ this becomes 
\[ F'_q(x) = \frac{1-x^q}{(x-1)^2}= - (x-1)^{q-2}.\]
We also have that $s_q(x) = (x-1)^q$ and $s'_q(x)=0$ over $\F_q$. Hence
\[
u_q(x) = s'_q(x)F_q(x)-F'_q(x)s_q(x)=(x-1)^{2q-2} \pmod q.
\qedhere
\]
\end{proof}
\begin{cor}
    The polynomial $u_q$ is irreducible. 
\end{cor}
\begin{proof}
 Let $v_q(x)=u_q(x+1).$ Then $v_q(x) \equiv x^{2q-2} \pmod q$ and $v_q(0)=u_q(1) = q(q-1)$. By Eisenstein's criterion for irreducibility, we conclude that $v_q$ (and hence $u_q$) is irreducible. 
 \end{proof}

\begin{prop} \label{prop:roots_over_fp}
\label{prop:qp_multiplicity_mod_p}
Suppose $n= pq$ for odd primes $p > q$. Let $x_0\in \overline{\F}_p$ be a root of $F_n$. Then
\begin{enumerate}[(a)]
    \item $\mult_{x_0}(F_n) - 1 = \mult_{x_0}(u_q)$.
    \item If $\disc(u_q) \neq 0 \mod p$, then $\mult_{x_0}(F_n) \leq 2$.
    \item Suppose $x_0 \in \F_p$. Then $\mult_{x_0}(F_n) - 1 
    = \mult_{x_0}(f_n)$.
\end{enumerate}
\end{prop}
\begin{proof}
Using the recursive formula from \cref{prop:recursive_F}, we have
\begin{align*}
    F_n(x) = \frac{x^{qp}-1}{x^q-1}F_q(x) - F_q(x^p),
\end{align*}
and hence
\begin{align*}
F_n(x)&\equiv (x^q-1)^{p-1}F_q(x) - F_q(x)^p \pmod p,\\
F'_n(x)&\equiv F'_q(x)(x^q-1)^{p-1}-qx^{q-1}F_q(x)(x^q-1)^{p-2}\pmod p \\ 
          &\equiv-(x^q-1)^{p-2} u_q(x) \pmod{p}.
\end{align*}

\begin{enumerate}[(a)]
    \item 
    \cref{prop:value_at_zeta_d} says that $F_n(\zeta_q)= -\varphi(p) \equiv 1 \pmod p$, and $F_n(1) = \varphi(n)=(q-1)(p-1) \equiv 1-q \not\equiv 0 \pmod p$. Therefore, if $x_0 \in \overline{\F}_p$ is a zero of $F_n$, then it is not a zero of $x^q-1$.
    Hence the relation of $F'_n$ and $u_q$ obtained above shows that $\mult_{x_0}(u_q)=\mult_{x_0}(F_n) - 1$.
    \item This is a straight-forward consequence of Part (a). If $\disc(u_q) \neq 0 \mod p$, then the reduction of $u_q$ modulo $p$ is separable. Hence $\mult_{x_0}(u_q) \leq 1$ and $\mult_{x_0}(F_n) \leq 2$.
    \item 
    Since $x_0 \in \F_p$, it is a $(p-1)$-th root of unity. Since $x_0$ is a zero of $F_n$, we know as in Part (a) that it is not a $q$-th root of unity. So there exists some $d$ dividing $p-1$, $d \neq q$, such that $x_0$ is a root of the $d$-th cyclotomic polynomial $\Phi_d$. Therefore by \cref{prop:cyclotomic_factors_pq} we get that $\mult_{x_0}(F_n) - 1 = \mult_{x_0}(f_n)$.
\end{enumerate}
\end{proof}
\begin{rem}
     We note that irreducibility of the polynomial $u_q \in \Z[x]$ implies in particular that $\disc(u_q) \neq 0$. Therefore for primes $p$ sufficiently large compared to $q$, we have $\disc(u_q) \neq 0 \pmod p$ and hence $\mult_{x_0} (F_{pq}) \leq 2$.
\end{rem}
To further study separability of $F_n$ over $\F_p$, we introduce the following auxiliary polynomials. Let 
\[ a(x,y)= s_q(x)- y t_q(x) ~\text{ where } s_q(x)=x^q-1, ~t_q(x)=F_q(x)=\sum_{i=1}^{q-1} x^i,\]
and let $R_q(y)$ be the resultant of $u_q(x)$ and $a(x,y)$ with respect to the variable $x$. 
There is a direct link between separability of $F_n$ and the arithmetic of $R_q$.

\begin{prop}
\label{prop:repeated_root_Res}
Suppose that $F_n$ has a repeated root $x_0 \in \overline{\F}_p.$ Then $R_q$ has a root $\mu \in \F_p$.
\end{prop}

\begin{proof}
By \cref{prop:qp_multiplicity_mod_p} Part (a), $\mult_{x_0}(u_q)=\mult_{x_0}(F_n)-1\geq 1$, i.e., $x_0$ is a root of $u_q$.  
We claim that $x_0$ is not a root of $F_q$. Suppose on the contrary that it is a root of $F_q$. Then firstly it is a simple root because $(x-1)F_q(x)=x(x^{q-1}-1)$ is separable over $\F_p$. Since $x_0$ is a repeated root of $F_n(x)=(x^q-1)^{p-1}F_q(x)-F_q(x)^p$, and only a simple root of $F_q(x)$, we deduce that $x_0$ must be a root of $x^q-1$. On the other hand, since $x_0 \neq 0$ and $x_0(x_0^{q-1}-1) = (x_0-1)F_q(x_0) = 0$, we get that $x_0^{q-1}-1$. This forces $x_0=x_0^q-1- x_0(x_0^{q-1}-1)+1=1$. But this is a contradiction because $F_n(1) = \varphi(n) = (p-1)(q-1) \neq 0 \pmod p$. 

Now 
$0=F_n(x_0)=(x_0^q-1)^{p-1} F_q(x_0)-F_q(x_0)^{p}$ implies that $(x_0^q-1)^{p-1}=F_q(x_0)^{p-1}$. Hence $x_0^q-1=\mu F_q(x_0)$ for some $\mu \in \F_p^\times$. Thus $x_0$ is a root of the polynomial $a(x,\mu) = x^q-1-\mu F_q(x)\in \F_p[x]$. In particular, $a(x,\mu)$ and $u_q(x)$ have a common root. Therefore
\[
{\rm resultant}(u_q(x),a(x,\mu))=R_q(\mu)=0.
\qedhere
\]
\end{proof}

\subsection{Further properties of the resultant $R_q(y)$}
We record in this section a few interesting properties of the resultant polynomial $R_q(y)$. We need the following lemma.
\begin{lem}
\label{lem:resultants_sq_tq}
    Let $q$ be an odd prime, and $s_q(x) = x^q-1$ and $t_q(x) = \sum\limits_{i=1}^{q-1} x^i$. Then
    \begin{enumerate}[(a)]
        \item $\Res(s_q(x), s_q'(x))=q^q.$
        \item $\Res(t_q(x), t_q'(x))=-(q-1)^{q-3}.$
        \item $\Res(t_q(x), s_q(x)) =q-1.$
    \end{enumerate}
\end{lem}
\begin{proof}
\begin{enumerate}[(a)]
    \item Let $\zeta_k$, $k=1,\ldots, q$, be all the $q$-th roots of unity. Then
    \[
    \Res(s_q(x),s'_q(x))=\prod_{k=1}^q s'_q(\zeta_k)=q^q\left(\prod_{k=1}^q \zeta_k\right)^{q-1}=q^q.
    \]
    \item Since $(x-1)t_q(x)=x^q-x$, we have
    \[
    {\rm disc}(x^q-x)= {\rm disc}(x-1){\rm disc}(t_q(x)) \Res(x-1,t_q(x))^2.
    \]
    Let $\zeta_k$, $k=1,\ldots,q-1$, be all the $(q-1)$-th roots of unity. Then
    \[
    \begin{aligned}
    {\rm disc}(x^q-x)&= (-1)^{q(q-1)/2}\Res(x^q-x,qx^{q-1}-1)\\
    &=(-1)^{q(q-1)/2}\cdot(-1)\cdot\prod_{k=1}^{q-1}(q \zeta_k^{q-1}-1) = -(-1)^{q(q-1)/2}(q-1)^{q-1}.
    \end{aligned}
    \]
    Also, we have  $\Res (x-1,t_q(x))^2=t_q(1)^2=(q-1)^2$. Hence
    \[
    \disc (t_q(x))=-(-1)^{q(q-1)/2}(q-1)^{q-3},
    \] and thus
    \[
    \Res(t_q(x),t'_q(x))=(-1)^{(q-1)(q-2)/2} {\rm disc}(t_q(x))= -(q-1)^{q-3}.
    \]
    \item Let $\zeta_k$, $k=1,\ldots, q$, be all the $q$-th roots of unity, where $\zeta_q=1$. Then
    \[
    \Res(s_q(x),t_q(x))=\prod_{k=1}^n t_q(\zeta_k)=(q-1)  \prod_{k=1}^{q-1} \dfrac{\zeta_k^q-\zeta_k}{\zeta_k-1}=(q-1)\prod_{k=1}^{q-1} \dfrac{1-\zeta_k}{\zeta_k-1}=q-1.
    \qedhere
    \]
\end{enumerate}
\end{proof}

\begin{prop}
Over $\F_q$, the polynomial $R_q(y)$ factors as $R_q(y)=y^{2q-2}$.
\end{prop}

\begin{proof}
Using the property that $\Res(AB, C)=\Res(A,C) \Res(B,C)$ and using \cref{lem:u_mod_q} which says that $u_q(x)=(x-1)^{2q-2}$ over $\F_q$, we have 
\begin{align*}
    R_q(y) = \Res(a(x,y), u_q(x)) &= \Res(a(x,y), (x-1)^{2q-2}) = \left( \Res(a(x,y), x-1) \right)^{2q-2} \\
    &=a(1,y)^{2q-2} = (q-1)^{2q-2} y^{2q-2}=y^{2q-2} \pmod q.
\qedhere
 \end{align*}
\end{proof}

\begin{prop}
    The polynomial $R_q(y)$ is an even polynomial of degree $2q-2$, i.e., it is a polynomial in $y^2$. It has leading coefficient $-(q-1)^{q-2}$ and constant coefficient $(q-1)q^{q}$.
\end{prop}
\begin{proof}
We first note that 
\begin{align*}
a\left(1/x, y \right) &= s_q \left(1/x \right)-y t_q \left(1/x \right) 
=-x^{-q} \left( s_q(x)+yt_q(x) \right). 
\end{align*}
and hence
$x^q a\left(1/x, y \right) a(x,y) = (y^2 t_q(x)^2-s_q(x)^2)$ is a polynomial in the variables $x, y^2$.
Let $z_1, z_2, \ldots, z_{2q-2}$ be the roots of $u_q$. Since $u_q$ is a reciprocal polynomial, we can assume further that $z_{i} z_{2q-1-i} =1$. Thus we have
\begin{align*}
    R_q(y) = \Res(a(x,y), u_q(x)) = \prod_{i=1}^{2q-2} a(z_i,y) =\prod_{i=1}^{q-1} \left( a(z_i,y) a\left(1/z_i, y\right) \right)
\end{align*}
which shows that $R_q$ is an even polynomial. The formula 
\[ R_q(y) = \prod_{i=1}^{2q-2} a(z_i,y) = \prod_{i=1}^{2q-2} (s_q(z_i)-y t_q(z_i))\]
also shows that the leading coefficient of $R_q$ is $\prod\limits_{i=1}^{2q-2} t_q(z_i) = \Res(t_q(x), u_q(x))$, and the constant coefficient of $R_q$ is $\prod\limits_{i=1}^{2q-2} s_q(z_i) =\Res(s_q(x), u_q(x)).$ We can compute them using \cref{lem:resultants_sq_tq}. We compute the leading coefficient as follows 
\begin{align*}
 \Res(t_q(x), u_q(x)) &= \Res(t_q(x), s_q'(x)t_q(x)-s_q(x)t_q'(x)) = \Res(t_q(x), -s_q(x)t_q'(x)) \\ 
 &=  \Res(t_q(x), s_q(x)) \Res(t_q(x), t_q'(x))= -(q-1)^{q-2}. 
\end{align*}
Similarly, the constant coefficient of $R_q$ is $(q-1)q^q.$
\end{proof}

\section{The case $n=3p$}
\label{sec:3p}

In this section, we focus on the special case $n=3p$ for some prime $p > 3$. We note that the set $S_{3p}$ described in \cref{eq:S_n} can be rewritten explicity. 
\[ S_{3p}=\begin{cases} 
 \{d\in \N\mid d>1, d\not=3, d\mid p-1\}\cup \{8\} &\text{ if } p\equiv 1\mod {12}\\
 \{d\in \N\mid d>1, d\mid p-1\}\cup \{8\} &\text{ if } p\equiv 5\mod {12}\\
  \{d\in \N\mid d>1,d\not=3, d\mid p-1\} &\text{ if } p\equiv 7\mod {12}\\
 \{d\in \N\mid d>1, d\mid p-1\} &\text{ if } p\equiv 11\mod {12}.
\end{cases} \]
Accordingly, we have the explicit description
\begin{align}
\label{eq:explicit_F3p}
\begin{split}
F_{3p}(x)&=f_{3p}(x)\cdot x\cdot \prod_{d\in S_{3p}}\Phi_d(x)
\\&=\begin{cases}
f_{3p}(x)x\dfrac{x^{p-1}-1}{x^3-1} &\text{ if } p\equiv 1,7,19\pmod {24}\\
f_{3p}(x)x\dfrac{x^{p-1}-1}{x^3-1}\Phi_8(x) &\text{ if } p\equiv 13\pmod {24}\\
f_{3p}(x)x\dfrac{x^{p-1}-1}{x-1}\Phi_8(x) &\text{ if } p\equiv 5\pmod {24}\\
f_{3p}(x)x\dfrac{x^{p-1}-1}{x-1} &\text{ if } p\equiv 11,17,23\pmod {24}.\\
\end{cases}
\end{split}
\end{align}
\begin{prop}
\label{prop:explicit_f3p}
Let $p > 3$ be a prime. Then the polynomial $f_{3p}$ has the explicit formula
\[
f_{3p}(x)=\begin{cases}
x^{2p+2}+x^{2p+1}+x^{p+2}+x^p+x+1 & \text{ if } p\equiv 1,7,19\pmod {24}\\
\dfrac{x^{2p+2}+x^{2p+1}+x^{p+2}+x^p+x+1}{x^4+1} & \text{ if } p\equiv 13\pmod {24}\\
\dfrac{x^{2p+2}+x^{2p+1}+x^{p+2}+x^p+x+1}{(x^2+x+1)(x^4+1)} & \text{ if } p\equiv 5\pmod {24}\\
\dfrac{x^{2p+2}+x^{2p+1}+x^{p+2}+x^p+x+1}{x^2+x+1} & \text{ if } p\equiv 11, 17, 23\pmod {24}.
\end{cases}
\]
and 
\[
\deg f_{3p}=\begin{cases}
2p+2 & \text{ if } p\equiv 1,7,19\pmod {24}\\
2p-2 & \text{ if } p\equiv 13\pmod {24}\\
2p-4 & \text{ if } p\equiv 5\pmod {24}\\
2p & \text{ if } p\equiv 11, 17, 23\pmod {24}.
\end{cases}
\]
\end{prop}
\begin{proof}
The statements follow from noting that
\begin{align*}
(x^3-1)F_{3p} &= x^{3p+2}+x^{3p+1}+x^{2p}+x^p- (x^{2p+3}+x^{p+3}+x^2+x)\\
&= (x^p-x)(x^{2p+2}+x^{2p+1}+x^{p+2}+x^p+x+1).
\end{align*}
\end{proof}
It is a classical theorem that for any distinct odd primes $p, q$, the coefficients of the cyclotomic polynomial $\Phi_{pq}$ are in $\{0, -1, 1 \}$ (see \cite{brookfield2016coefficients}). The smallest integer $n$ such that $\Phi_n$ has a coefficient not contained in $\{0,-1, 1\}$ is $n=105$. Along these lines, we observe that the coefficients of $f_{3p}$ are quite small. In fact, for $p<1200$, we verify using SageMath that the coefficients of $f_{3p}$ are in the set $\{-2, -1, 0 , 1, 2 \}$. In fact this is true for all $p$ as the next theorem shows.

\begin{thm}
\label{thm:coeff_3p}
Let $p > 3$ be a prime. The coefficients of $f_{3p}$ are in the set $\{-2, -1, 0, 1, 2 \}$.
\end{thm}

\begin{proof}
Suppose that $p \equiv 1 \pmod{3}$. If $p\equiv 1,7,19\pmod {24}$, the statement is clear from \cref{prop:explicit_f3p}.
If $p\equiv 13\pmod {24}$, say $p=13+24a$ for some $a\in \N$, then
\[
\begin{aligned}
x^{2p+2}+1 &= (x^4)^{7+12a}+1=(x^4+1)\sum_{k=0}^{6+12a}(-1)^k x^{4k}\\
x^{2p+1}+x^{p+2} &= x^{p+2} [(x^4)^{3+6a}+1] = (x^4+1) \sum_{k=0}^{2+6a}(-1)^k x^{4k+15+24a}\\
x^{p}+x&= x[(x^4)^{3+6a}+1]=(x^4+1)\sum_{k=0}^{2+6a}(-1)^k x^{4k+1}.
\end{aligned}
\]
Hence 
\[
f_{3p}(x)=\sum_{k=0}^{6+12a}(-1)^k x^{4k} +\sum_{k=0}^{2+6a}(-1)^k x^{4k+15+24a}+\sum_{k=0}^{2+6a}(-1)^k x^{4k+1}
\]
which shows that all the coefficients of $f_{3p}$ are in $\{-1,0,1\}$.

Now suppose that $p \equiv 2 \pmod 3$. Write $p=2+3a$ for some $a\in \N$. Let 
\[g(x)=\sum_{k=a+1}^{2a+1}x^{3k+1}-\sum_{k=a+1}^{2a}x^{3k+2}+\sum_{k=0}^{a}x^{3k}-\sum_{k=0}^{a-1}x^{3k+2}.\]
It is straightforward to check that
\[
\begin{aligned}
(x^2+x+1)g(x)&=x^{6a+6}+x^{6a+5}+x^{3a+4}+x^{3a+2}+x+1\\
&=x^{2p+2}+x^{2p+1}+x^{p+2}+x^p+x+1.
\end{aligned}
\]
Now if $p\equiv 11, 17, 23\pmod {24}$, then \cref{prop:explicit_f3p} says that $f_{3p}(x)=g(x)$, and so all the coefficients of $f_{3p}$ are in $\{-1,0,1\}$.
The last remaining case is $p\equiv 5\pmod {24}$. If we write $g(x)=\sum\limits_{k=0}^{2p}b_kx^k$, then the coefficients $b_k$ are given by
\[
b_k= \begin{cases} 1 &
\text{ if $k\equiv 1\pmod 3$ and $p+2\leq k\leq 2p$}\\
-1 &\text{ if $k\equiv 2\pmod 3$, ~$k \leq 2p$ and $k\not=p$}\\
1& \text{ if $k\equiv 0\pmod 3$ and $0\leq k\leq p-2$}\\
0 &\text{ otherwise}
\end{cases}
\]
In particular we have $b_k=b_{k'}$ whenever $k\equiv k'\pmod 3$ and  $0\leq k,k'\leq p-1$. 
\cref{prop:explicit_f3p} says that in this case we have $f_{3p}(x)(x^4+1)=g(x)$. If we write $f_{3p}(x)=\sum\limits_{k=0}^{2p-4}a_kx^k$, then we get that the coefficients of $g$ and $f_{3p}$ are related by
\begin{align*}
b_k = \begin{cases} a_k \quad &\text{ if } \quad 0 \leq k \leq 3\\
a_k+a_{k-4} \quad & \text{ if } \quad 4 \leq k \leq 2p-4\\
a_{k-4} \quad &\text{ if } \quad 2p-3 \leq k \leq 2p. \end{cases}
\end{align*}
For $24 \leq k \leq p-1$, using the recursive formula $b_k = a_k + a_{k-4}$ repeatedly, we have
\begin{align*}
a_{k} - a_{k-24} &= b_{k} - b_{k-4} + b_{k-8} - b_{k-12} + b_{k-16} - b_{k-20}\\
&=(b_{k}-b_{k-12})-(b_{k-4}-b_{k-16})+(b_{k-8}-b_{k-20})=0.
\end{align*}
showing that the sequence $a_0,a_1,\ldots,a_{p-1}$ is periodic with a period $24$. It is straightforward to check that the sequence $a_0,a_1,\ldots,a_{23}$ is
\[
1,0,-1,1,-1,-1,2,-1,0,2,-2,0,1,-2,1,1,-1,1,0,-1,0,0,0,0.
\]
Hence $a_k\in \{-2,-1,0,1,2\}$ for $0\leq k \leq p-1$. Since $f_{3p}$ is reciprocal, we know that $a_{k}=a_{2p-4-k}$ and so $a_k\in \{-2,-1,0,1,2\}$ for all $0 \leq k \leq 2p-4$.
\end{proof}

\begin{cor}
Let $a_{\frac{\deg f_{3p}}{2}}$ be the middle coefficient of $f_{3p}$. Then
\[
a_{\frac{\deg f_{3p}}{2}}
=\begin{cases}
0 & \text{ if } p\equiv 1,7,11,17,19,23 \pmod {24}\\
1 & \text{ if } p\equiv 5 \pmod {24}\\
-1 & \text{ if } p\equiv 13 \pmod {24}.
\end{cases}
\]
\qed
\end{cor}

Next, we study more properties of $f_{3p}$ in characteristic $p$. We start with the following theorem which is an improved version of \cref{prop:roots_over_fp} in the special case $n = 3p$. 

\begin{thm}
\label{thm:3p_multiplicity_mod_p}
Let $p>3$ be a prime.
Let $x_0\in \overline{\F}_p$ be a root of $F_{3p}$.
\begin{enumerate}
    \item The multiplicity of $x_0$ as a root of $F_{3p}$ is at most $2$.
    \item The multiplicity of $x_0$ as a root of $F_{3p}$ is $2$ if and only if $x_0 \in \F_p$ and $x_0$ is a root of 
    $u_3(x)=x^4+2x^3+2x+1$.
\end{enumerate}
\end{thm}
\begin{proof}
We have 
${\rm disc}(u_3)=-1728=-2^6 \times 3^3 \neq 0 \pmod{p}$.
So the first statement follows from \cref{prop:qp_multiplicity_mod_p} Part (b).

The `if' part of the second statement follows from \cref{prop:qp_multiplicity_mod_p} Part (a). Now we discuss the `only if' part. Suppose that the multiplicity of $x_0$ as a root of $F_{3p}$ is $2$. \cref{prop:qp_multiplicity_mod_p} Part (a) immediately implies that $x_0$ is a root of $u_3$. So the only thing left to show is that $x_0 \in \F_p$.

As described in the proof of \cref{prop:repeated_root_Res}, there exists $\mu\in \F_p$ such that 
\[
a_3(x_0,\mu) = 0, \quad \text{and} \quad R_3(\mu) = {\rm resultant}(a_3(x,\mu),u_3(x))=-2\mu^4+36\mu^2+54=0.
\]
The latter equation implies that $108=(\mu^2-9)^2$ and hence $3$ is a square modulo $p$. Write $3=c^2$ for some $c\in \F_p$. We have
\[
u_3(x)=(x^2+x+1)^2-3x^2=(x^2+(1+c)x+1)(x^2+(1-c)x+1) \in \F_p[x].
\]
Let $b(x)\in \F_p[x]$ be the minimal polynomial of $x_0$ over $\F_p$. Then $b(x)$ is an irreducible factor of both $u_3(x)$ and $a_3(x,\mu)$. In particular $\deg b(x)=1$ or $2$. 

If $\deg b(x)=2$, then $b(x)=x^2+(1 \pm c)x+1$ is reciprocal. Hence the zeroes of $b(x)$ are $x_0$ and $1/x_0$. Since $b(x)$ divides $a_3(x,\mu)=x^3-\mu x^2-\mu x-1$, the zeroes of $a_3(x,\mu)$ are $x_0, 1/x_0$ and $\beta$, for some $\beta\in \overline{\F}_p$. By Vieta's formula, $\beta = x_0 \cdot (1/x_0) \cdot \beta=1$. Hence $0=a_3(1,\mu)=-2\mu$, a contradiction since $R_3(0) = 54 \neq 0 \in \F_p$.

The above argument shows that $\deg b(x)=1$ and hence $x_0\in \F_p$.
\end{proof}

\begin{cor} Let $p>3$ be a prime. Then 
${\rm disc}(F_{3p})=0\pmod p$ if and only if $u(x)=x^4+2x^3+2x+1$ has a root modulo $p$. In particular,
\begin{enumerate}[(a)]
    \item if $p\equiv \pm5\pmod {12}$ then $p\nmid {\rm disc}(F_{3p})$, 
    \item if $p\equiv 11\pmod{12}$ then $p\mid {\rm disc}(F_{3p})$,
    \item if $p\equiv 1\pmod {12}$ then $p\mid {\rm disc}(F_{3p})$ if and only if $12$ is a quartic residue mod $p$.
\end{enumerate}
\end{cor}
\begin{proof}
The first statement follows immediately from \cref{thm:3p_multiplicity_mod_p}. 
In particular if  $p\equiv \pm5\pmod {12}$ then $\left( \dfrac{3}{p} \right)=-1$ and hence $u_3(x)=(x^2+x+1)^2-3x^2$ has no root in $\F_p$. Therefore ${\rm disc}(F_{3p}) \neq 0\pmod p$.

Now suppose $p\equiv \pm 1\pmod {12}$. Then $\left( \dfrac{3}{p} \right)=1$. Let $c\in \F_p$ such that $3= c^2$.
\[
u_3(x)=(x^2+x+1)^2-3x^2=(x^2+(1+c)x+1)(x^2+(1-c)x+1).
\]
The discriminant of $x^2+(1+c)x+1$ is equal to $(1+c)^2-4=2c$, and the discriminant of $x^2+(1-c)x+1$ is equal to $(1-c)^2-4 = -2c$.

If $p\equiv 11\pmod {12}$ then $\left( \dfrac{-1}{p} \right) = -1$. Hence either $2c$ or $-2c$ is a square in $\F_p$. 
Therefore, either $x^2+(1+c)x+1$ or $x^2+(1-c)x+1$ has a root in $\F_p$. So \cref{thm:3p_multiplicity_mod_p} implies that $F_{3p}$ has a double root in $\F_p$, and so $p\mid {\rm disc}(F_{3p})$.

If $p\equiv 1\pmod{12}$ then $\left(\dfrac{2c}{p}\right)=\left(\dfrac{-2c}{p}\right)$. By \cref{thm:3p_multiplicity_mod_p}, we get that $p\mid {\rm disc}(F_{3p})$ if and only if there exists $a\in \F_p$ such that $a^2=2c$. Since $c^2 = 3$, this is true if and only if there exists $a\in \F_p$ such that $a^4=12$, i.e., $12$ is a quartic residue mod $p$.
\end{proof}
\begin{cor}
Let $x_0 \in \F_p$. Then $x_0$ is a root of the Fekete polynomial $f_{3p}$ if and only if it is a root of $u_3(x)= x^4+2x^3+2x+1.$
\end{cor}
\begin{proof}
    Suppose $x_0 \in \F_p$ is a root of $u_3(x)$. Then \cref{thm:3p_multiplicity_mod_p} says that $x_0$ is a double root of $F_{3p}$. By the explicit description of the cyclotomic factors of $F_{3p}$ given in \cref{eq:explicit_F3p}, it is clear that $x \prod\limits_{d \in S_{3p}} \Phi_d(x)$ is separable over $\F_p$. So we can deduce that $x_0$ is a root of $f_{3p}$.
    
    Conversely, suppose $x_0 \in \F_p$ is a root of $f_{3p}$. So it is also a root of $F_{3p}$. We know by \cref{prop:value_at_zeta_d} that $F_{3p}(1) = 2(p-1) = -2 \neq 0 \in \F_p$ and $F_{3p}(\zeta_3) = -(p-1) = 1 \neq 0 \in \F_p$.
    This means that $x_0^3 \neq 1$. On the other hand, since $x_0 \in \F_p$, we know that $x_0^{p-1} = 1$. Therefore by \cref{eq:explicit_F3p}, we see that $x_0$ has multiplicity at least $2$ as a root of $F_{3p}$. \cref{thm:3p_multiplicity_mod_p} then gives us that $x_0$ is a root of $u_3(x)$.
\end{proof}
\begin{cor}
The polynomial $f_{3p}$  is separable over $\F_p$, and hence also separable over $\Z$. Consequently, $g_{3p}$ is separable as well. 
\end{cor}
\begin{proof}
    Suppose on the contrary that $f_{3p}$ is not separable over $\F_p$. Let $x_0 \in \overline{\F}_p$ be a root of $f_{3p}$ with $\mult_{x_0}(f_{3p}) \geq 2$. Then $\mult_{x_0}(F_{3p}) \geq 2$ and \cref{thm:3p_multiplicity_mod_p} says that $\mult_{x_0}(F_{3p}) = 2$ and $x_0 \in \F_p$. This is absurd because $x_0$ is also a root of $x \prod\limits_{d \in S_{3p}} \Phi_d(x)$ by \cref{eq:explicit_F3p}, but that makes $\mult_{x_0}(F_{3p}) \geq 3$.
\end{proof}

We now write down explicitly the values $f_{3p}(1)$ and $f_{3p}(-1)$. These can be obtained directly from \cref{prop:values_at_1and-1}, or computed using the explicit formula for $f_{3p}$ in \cref{prop:explicit_f3p}. We then use this information to prove a fact about $\disc(f_{3p})$ that was first discovered by experimental data. Note that $\disc(f_{3p}) \neq 0$ since $f_{3p}$ is separable.
\begin{lem}
\label{lem:f3p_value_1andminus1}
Let $p > 3$ be a prime. Then
\[f_{3p}(1)=
\begin{cases}
6 &\text{ if } p\equiv 1,7,19\pmod {24}\\
3 &\text{ if } p\equiv 13\pmod {24}\\
1  &\text{ if } p\equiv 5\pmod {24}\\
2&\text{ if } p\equiv 11,17,23\pmod {24}.
\end{cases}
f_{3p}(-1)=
\begin{cases}
-1 &\text{ if } p\equiv 5,13\pmod {24}\\
-2 &\text{ otherwise.}
\end{cases}\] 
\end{lem}
\begin{prop}
Let $p > 3$ be a prime.
If $p \equiv 1 \pmod{3}$ then $\disc(f_{3p})<0$.
If $p \equiv 2 \pmod{3}$ then $\disc(f_{3p})$ is a nonzero perfect square.
\end{prop}
\begin{proof}
    Since $f_{3p}$ is a reciprocal polynomial with $g_{3p}$ its trace polynomial, we have
    \[ \disc(f_{3p})= (-1)^{\frac{\deg(f_{3p})}{2}} f_{3p}(-1) f_{3p}(1) \disc(g_{3p})^2. \]
    Calculating using \cref{lem:f3p_value_1andminus1} and \cref{prop:explicit_f3p}, we get
    \begin{align*}
    \disc(f_{3p})&=\begin{cases}
    (-1)^{p+1}\cdot(-2)\cdot6\cdot\disc(g_{3p})^2 &\text{ if } p
    \equiv 1,7,19\pmod {24}\\
     (-1)^{p-1}\cdot(-1)\cdot3\cdot\disc(g_{3p})^2 &\text{ if } p\equiv 13\pmod {24}\\
      (-1)^{p-2}\cdot(-1)\cdot1\cdot\disc(g_{3p})^2 &\text{ if } p\equiv 5\pmod {24}\\
       (-1)^{p}\cdot(-2)\cdot2\cdot\disc(g_{3p})^2 &\text{ if } p\equiv 11,17,23\pmod {24}
    \end{cases}\\
    &=\begin{cases}
    -12\disc(g_{3p})^2 &\text{ if } p\equiv 1,7,19\pmod {24}\\
     -3\disc(g_{3p})^2 &\text{ if } p\equiv 13\pmod {24}\\
      \disc(g_{3p})^2 &\text{ if } p\equiv 5\pmod {24}\\
       4\disc(g_{3p})^2 &\text{ if } p\equiv 11,17,23\pmod {24}.
    \end{cases}
    \end{align*}
\end{proof}

\section{The case $n=5p$}
\label{sec:5p}

In this section, we provide some partial results for the case $n=5p$ where $p$ is a prime greater than $5$. The goal is to prove the following analog of \cref{thm:3p_multiplicity_mod_p}. 

\begin{thm}
\label{thm:5p_multiplicity_mod_p}
Let $p>5$ be a prime.
Let $x_0\in \overline{\F}_p$ be a root of $F_{5p}$. 
\begin{enumerate}[(a)]
    \item The multiplicity of $x_0$ as a root of $F_{5p}$ is at most $2$.
    \item The multiplicity of $x_0$ as a root of $F_{5p}$ is $2$ if and only $x_0 \in \F_p$ and $u_5(x_0) = 0$.
\end{enumerate}
\end{thm}
    


\begin{proof}[Proof of \cref{thm:5p_multiplicity_mod_p} Part (a)]

By \cref{prop:roots_over_fp} Part (b), if the discriminant of $u_5$ is not zero modulo~$p$, we get the desired result. We compute that
\[ \disc(u_5)= - 2^{12} \cdot 5^7 \cdot 11^2. \]
When $p=11$, we check directly that $F_{5p} = F_{55}$ has no repeated root in $\Fbar_{11}$.
\end{proof}


In this section, we use $a_\mu(x)$ for $a_5(x,\mu)=(x^5-1)-\mu(x+x^2+x^3+x^4)$.
Using SageMath, we see that the resultant of $a_{\mu}(x)$ and $u_5(x)$ is given by 
\[ R_5(\mu) = \Res(a_\mu(x), u_5(x))= -(64\mu^8 - 400 \mu^6 - 500\mu^4 - 25000\mu^2 - 12500).\] 
Because $a_{\mu}(x)$ and $u_5(x)$ have a common root, their resultant must be $0$. In other words, we know that $\mu \in \F_p$ is a root of $R_5(y).$ Using SageMath, we can see that we can rewrite $R_5(y)$ in the following form 
\[ -R_5(y) = (8y^4-25y^2+125)^2-5(25y^2+75)^2 .\] 
\begin{lem}
\label{lem:5p_repeated_root_modp}
Suppose that $F_{5p}(x)$ has a repeated root $x_0 \in \overline{\F}_p$, then $\left(\frac{5}{p} \right)=1.$    
\end{lem}

\begin{proof}
    As explained above, the existence of a repeated root $x_0 \in \overline{\F}_p$ implies that $R_5(y)$ has a root $\mu \in \F_p$ where 
    \[ -R_5(y) = (8y^4-25y^2+125)^2-5(25y^2+75)^2 .\] 
If $25\mu^2+75 \neq 0$, then we conclude that $\left(\frac{5}{p} \right)=1.$ Otherwise, we must have $\mu^2+3=0$. Consequently 
\[ 0 = -R_5(\mu) =(8 \mu^4-25 \mu^2+125)^2 = 2^4 \times 17. \]

Since $p>5$, we conclude that $p=17.$ This is impossible because $\left(\frac{-3}{17} \right)=-1$, and hence the equation $\mu^2+3=0$ has no solution in $\F_p.$
\end{proof}

\begin{cor}
 If $\left(\frac{5}{p} \right)=-1$ then $F_{5p}$ is separable over $\F_p$. \qed 
\end{cor}

We now complete the proof of \cref{thm:5p_multiplicity_mod_p}.

\begin{proof}[Proof of \cref{thm:5p_multiplicity_mod_p} Part(2)]

Suppose $x_0\in \F_p$ is a root of $u_5(x)$. By \cref{def:uq}, since we know that $u_5(1) = q(q-1) \neq 0 \in \F_p$ and $u_5(\zeta_5) = -5\zeta_5^4 \neq 0 \in \overline{\F}_p$ for any primitive $5$-th root of unity $\zeta_5 \in \overline{\F}_p$, we get that $x_0^5 \neq 0$. On the other hand $x_0^{p-1} = 1$. Hence $F_{5p}(x_0) = 0$. So \cref{prop:qp_multiplicity_mod_p} Part (a) says that $\mult_{x_0}(F_{5p}) \geq 2$. Combining with \cref{thm:5p_multiplicity_mod_p} Part(a), we conclude that $\mult_{x_0}(F_{5p})= 2$.

Conversely suppose that $x_0 \in \overline{\F}_p$ is a double root of $F_{5p}$. By \cref{lem:5p_repeated_root_modp}, one has
$\left(\frac{5}{p} \right)=1.$ Let $c \in \F_p$ be such that $c^2=5$. Then we have
\[
\begin{aligned}u_5(x)&=(1+x+x^2+x^3+x^4)^2-5x^2\\
&= (1+x+x^2+x^3+x^4-cx^2)(1+x+x^2+x^3+x^4+cx^2).
\end{aligned}
\]

Let $m(x)$ be the minimal polynomial of $x_0$ over $\F_p$. Then $m(x)$ is a common divisor of $u_5(x)$ and $a_{\mu}(x).$ Up to a choice of $c$, we can assume that $m(x)$ is a divisor of 
\[ v(x)= 1+x+x^2+x^3+x^4-cx^2.\]
By polynomial division, we see that $a_{\mu}(x) = (x-\mu-1)v(x) + w(x)$, where $w(x) = cx^3-(c+c\mu)x^2+\mu$. Since $x_0$ is a common root of $v(x)$ and $a_\mu(x)$, we get that $x_0$ is a common root of $v(x)$ and $w(x)$.
 Hence $\deg m\leq 2$. 
Suppose that $\deg m=2$ and $m(x)=x^2+ax+b$, for some $a,b\in \F_p$.

\medskip
\noindent{\bf Case 1:} $b=1$, i.e., $m(x)$ is reciprocal. In this case, the roots of $m(x)$ are $x_0$ and $1/x_0$. This implies that $1/x_0$ is also a root of $a_\mu(x)$. Hence
\[
0=x_0^5a_\mu(1/x_0)= (1-x_0^5)-\mu (x_0+x_0^2+x_0^3+x_0^4).
\]
From $a_\mu(x_0)=0$, we see that $x_0^5-1=\mu (x_0+x_0^2+x_0^3+x_0^4)=1-x_0^5$. Hence  $x_0^5-1=0$. Thus 
$0=(x_0^5-1)=(1+x_0+x_0^2+x_0^3+x_0^4)(x_0-1)=cx_0^2(x_0-1)$. This implies that $x_0=0$ or $1$, a contradiction.

\medskip
\noindent{\bf Case 2:} $b\not=1$, i.e., $m(x)$ is not reciprocal. In this case, since $v(x)$ is reciprocal of degree 4, one has
\[
v(x)=\dfrac{1}{b} (x^2+ax+b)(1+ax+bx^2).
\]
By comparing the corresponding coefficients, we obtain 
\[
\dfrac{a+ab}{b}=1 \quad \text{ and } \dfrac{1+a^2+b^2}{b}=1-c.
\]
Hence
\[
a+ab=b \text{ and } 1+a^2+b^2=b-bc.
\]
Also, since $m(x)=x^2+ax+b$ is a divisor of $w(x)=cx^3-(c+c\mu)x^2+\mu$, one can write 
\[
cx^3-(c+c\mu)x^2+\mu =(x^2+ax+b)(cx-d)=cx^3+(ac-d)x^2+(bc-ad)x-bd,
\]
for some $d\in \F_p$. By comparing the corresponding coefficients, we obtain that
\[
ac-d=-c-c\mu,\quad bc-ad=0,\quad \text{ and } -bd=\mu.
\]
Hence
\[
bc=ad=a(ac+c+c\mu)=a^2c+ac+ac\mu.
\]
Thus $b=a^2+a+a\mu$. Also, we have
\[
a\mu=-abd=-b^2c.
\]
Hence 
\[
b=a^2+a-b^2c.
\]
In summary, we obtain the following three relations
\[
a+ab=b \;(1),\quad 1+a^2+b^2=b-bc \;(2), \quad b=a^2+a-b^2c \;(3).
\]
From $(2)$ and $(3)$ we get 
\[ b+a^2b+b^3=b^2-b^2c=b^2+b-(a^2+a).\]
Hence 
\[
b^2-b^3=a^2+a^2b+a=ab+a=b.
\]
(For the second and last equality, we use (1).) Since $b\not=0$, we obtain that  $b^2-b+1=0$.

Now from (2), we have $-bc=1+a^2+b^2-b=a^2$. Hence $a^4=b^2c^2=5b^2$. From (1), we obtain $b=a(1+b)$. Hence $b^4=a^4(1+b)^4=5b^2(1+b)^4$. Thus 
\[
b^2=5(1+b)^4=5(1+2b+b^2)^2=5(3b)^2=45b^2.
\]
We obtain that $p\mid 44$. Hence $p=11$. But we can check directly that $F_{5\cdot 11}(x)$ has no repeated root in $\overline{\F}_p$, a contradiction.
\end{proof}

\begin{cor}
The polynomial $f_{5p}$  is separable over $\F_p$, and hence also separable over $\Z$. Consequently, $g_{5p}$ is separable as well.
\end{cor}
\begin{rem}
One may wonder if a similar statement as \cref{thm:5p_multiplicity_mod_p} holds for general $n=pq$. It turns out that the answer is no. We provide some concrete counterexamples. 
\begin{enumerate}
\item When $q=7, p=101$, the polynomial $x^2+42x+10$ is an irreducible factor of $F_{pq}$ (and $f_{pq}$) over $\F_p$, with multiplicity equal to $2$.
\item When $q=11, p=13$, the polynomial $x^2+9x+10$ is an irreducible factor of $F_{pq}$ (and $f_{pq}$) over $\F_p$, with multiplicity equal to $2$.
\item When $q=11, p=61$, the polynomial $x^2+16x+14$ is an irreducible factor of $F_{pq}$ (and $f_{pq}$) over $\F_p$, with multiplicity equal to $2$.
\end{enumerate}
It would be interesting to investigate this problem further. For instance, can we get upper bounds on the multiplicity of a repeated root $x_0 \in \overline{\F}_p$ of $F_n(x)$?
\end{rem}

\section{Irreducibility of $f_n$}
\label{sec:irreducibility}

Checking irreducibility of a polynomial $f(x) \in \Z[x]$ is computationally expensive when $\deg (f)$ is large. A natural way to show irreducibility is by finding a prime $p \nmid \disc(f)$ such that the reduction $\fbar(x)$ mod $p$ is irreducible over $\F_p$. This strategy fails if $\deg (f) = 2n$ is even and $\disc (f)$ is a perfect square: For every $p \nmid \disc (f)$, there is a Frobenius element $\Frob_p$ in the symmetric group $S_{2n}$ determining the Galois orbit decomposition of the roots of $\fbar$. If $\disc (f)$ is a perfect square, $\Frob_p$ is an even permutation and hence cannot be a $2n$-cycle.

In this section, we will discuss some methods to computationally verify irreducibility of the Fekete polynomial $f_n \in \Z[x]$ by exploiting the fact that it is reciprocal. These are slight improvements of the criteria described in \cite{irreducibility}, and they all depend on irreducibility of the trace polynomial $g_n$, for which we turn to the method mentioned above.

We first make the following definition and state a key lemma from \cite{irreducibility}.
\begin{definition}
For a polynomial $h$ of degree $n$, its reversal polynomial is defined by $h_{\rev}(x) = x^n h(1/x)$.
\end{definition}
\begin{lem} \label{lem:irreducible_test}
Let $f$ be a monic reciprocal polynomial of degree $2n$. Let $g$ be the trace polynomial of $f$. Suppose that $g$ is irreducible. If $f$ is reducible, then there exists a monic irreducible polynomial $h \in \Z[x]$ such that 
$f(x) = \pm h(x) h_{\rev}(x)$.
Furthermore, if $f(1)>0$ then $f(x) = h(x) h_{\rev}(x)$.
\end{lem}    
\begin{proof}
See \cite[Corollary 6]{irreducibility} and the remark after it.
\end{proof}

Our first irreducibility criterion is \cref{prop:irreducibility1}. Together with \cref{prop:values_at_1and-1}, this shows that whenever the degree of the Fekete polynomial $f_{pq}$ is divisible by $4$, it is irreducible if and only if its trace polynomial $g_{pq}$ is irreducible.
\begin{prop}
\label{prop:irreducibility1}
Let $f$ be a monic reciprocal polynomial of degree $4n$. Let $g$ be the trace polynomial of $f$. Suppose that $g$ is irreducible and $f(1)f(-1)<0$. Then $f$ is irreducible. 
\end{prop}
\begin{proof}
Suppose $f$ is reducible. Then by \cref{lem:irreducible_test}, $f(x)= \pm x^{2n} h(x)h(\frac{1}{x})$ for some $h \in \Z[x]$. Consequently 
$f(1)f(-1) = h(1)^2 h(-1)^2 \geq 0$ which contradicts the hypothesis.
\end{proof}

Our second irreducibility criterion is \cref{prop:irreducibility2}. Together with \cref{thm:coeff_3p}, this proves that for all primes $p > 3$, the Fekete polynomial $f_{3p}$ is irreducible if and only if its trace polynomial $g_{3p}$ is irreducible.
\begin{prop}
\label{prop:irreducibility2}
Let $f = \sum_{k=0}^{2n} a_{k}x^k $ be a monic reciprocal polynomial of degree $2n$ such that $f(1)f(-1) \neq 0$. Let $g$ be the trace polynomial of $f$. Suppose that $g$ is irreducible and the middle coefficient of $f$ satisfies $|a_n| \leq 2.$ Then $f$ is irreducible. 
\end{prop}
\begin{proof}
Suppose that $f$ is reducible. Without loss of generality, we can assume $f(1)>0$ and hence by \cref{lem:irreducible_test}, $f(x) = h(x) h_{\rev}(x)$ for some monic irreducible polynomial $h \in \Z[x]$. Let $h(x) = \sum_{k=0}^n c_kx^k$ with $c_n=1$. Since $f$ is monic, by comparing leading coefficients, we deduce that $c_0=1$ as well. Further, we get the following expression for the middle coefficient of $f$:
\[ a_n = \sum_{k=0}^n c_k ^2.\]
Since $|a_n| \leq 2$ and $c_n=c_0=1$, we conclude that $c_k=0$ for $1 \leq k \leq n-1$ and thus $h(x)=x^n+1.$ Since $f(-1) \neq 0$, we deduce that $n$ is even. So $f = h h_{\rev}$ is a factorisation into reciprocal polynomials of even degree, which yields a non-trivial factorisation of $g$ (see \cite[Proposition 8]{irreducibility}). This contradicts the hypothesis that $g$ is irreducible.
\end{proof}
\begin{rem}
For $p \leq 1000$, we have also checked that the middle coefficient of $f_{5p}$ lies in the set $\{-2, -1, 0, 1, 2 \}$. However the middle coefficient of $f_{7 \times 73}$ is $-3$. \end{rem}

Our last irreducibility criterion is \cref{prop:counting_criterior}. This forms the basis of \cref{algo:irreducible}.
\begin{prop} \label{prop:counting_criterior}
 Let $f$ be a monic reciprocal polynomial of degree $2n$. Let $g$ be the trace polynomial of $f$. Suppose that $g$ is irreducible. Suppose also that there exists a prime number $q$ and a positive integer $m$ such that the number of irreducible factors of degree $m$ of the reduction $\fbar \in \F_{q}[x]$, counted with multiplicity, is  an odd number. Then $f$ is irreducible.
\end{prop}
\begin{proof}
    By \cref{lem:irreducible_test}, if $f$ is reducible then $f(x)= \pm h(x) h_{\rev}(x)$. So any irreducible factor $a(x)$ of $\fbar(x)$ divides either $\hb(x)$ or $\hb_{\rev}(x)$, and thus $a_{\rev}(x)$ is another irreducible factor of $\fbar(x)$. Therefore, the degree of any irreducible factor of $\fbar(x)$ must arise an even number of times. This contradicts the given hypothesis. So $f$ must be irreducible.
\end{proof}
\begin{algo}[Verifying irreducibility of a reciprocal polynomial $f(x)$ of even degree] \label{algo:irreducible}
\hfill
\begin{enumerate}[label={\bf Step \arabic*.}]
\item Compute the trace polynomial $g(x)$.
\item Find a prime number $q_1$ such that the reduction $\gbar(x) \in \F_{q_1}[x]$ is irreducible. 
\item Find a prime number $q_2$ such that the factorisation of $\fbar(x) \in \F_{q_2}[x]$ has an odd number of irreducible factors of some degree $m$ (when counted with multiplicity) i.e. satisfying the hypothesis of \cref{prop:counting_criterior}.
\item If Steps $2$ and $3$ terminate, return \texttt{True}.
\end{enumerate}
\end{algo}
\begin{expl}
We give a concrete example to demonstrate this method. Let us consider the Fekete polynomial $f_{15}(x)=x^6 - x^4 + x^3 - x^2 + 1$. Its trace polynomial is $g_{15}(x) = x^3 - 4x + 1$.
We can check that the reduction $\gbar_{15}(x) \in \F_3[x]$ is irreducible, so $g_{15}(x)$ is irreducible over $\Z$. Furthermore, the reduction $\fbar_{15}(x) \in \F_2[x]$ factors as: 
\[ \fbar_{15}(x) = (x^2 + x + 1)  (x^4 + x^3 + x^2 + x + 1). \]
Since the degree $2$ (and $4$) irreducible factor appears only once in the factorisation, by \cref{prop:counting_criterior}, the Fekete polynomial $f_{15}(x)$ is irreducible. 
\end{expl}

Using \cref{algo:irreducible}, we have verified irreducibility of the Fekete polynomial $f_n(x)$ when $n < 10^4$ is a product of two distinct odd primes.
For each $n$, our GitHub repository \cite{fekete} lists the corresponding prime pairs $(q_1, q_2)$. We are tempted to make the conjecture:
\begin{conj} \label{conj:irr}
    Let $n=pq$ be a product of two distinct odd primes. Then the Fekete polynomial $f_n$ from \cref{defn:fn} is irreducible. 
\end{conj}
The polynomials $1 - x + x^k - x^{2k-1} + x^{2k}$ associated to certain symmetric numerical semigroups were studied in \cite{Moree21}. In particular, besides determining the cyclotomic factors, \cite[Conjecture 6.6]{Moree21} predicts that these polynomials have a unique irreducible non-cyclotomic factor. We observe that the Fekete polynomials $F_{pq}$ considered in this paper do not come from numerical semigroups, but their factorisation exhibits an analogous picture, as described in \cref{thm:cyclotomic_factors_pq1}, \cref{defn:fn} and \cref{conj:irr}.

\section{Galois groups of $f_n$ and $g_n$}
\label{sec:Galois}

Let $f$ be a polynomial over $\Q$ of degree $m$. The Galois group of $f$ is the Galois group of its splitting field, denoted $\Q(f)$, as an extension of $\Q$. Using the permutation action on the roots of $f$, it is naturally a subgroup of the symmetric group $S_m$. If $f$ is irreducible, then it is a transitive subgroup of $S_m$. We recall the well-known result that a transitive subgroup of $S_m$ containing a $2$-cycle (transposition) and an $(m-1)$-cycle is all of $S_m$. Using this result, we have a criterion to verify that the Galois group of a polynomial of degree $m$ is the full symmetric group $S_m$. See \cite[Propisition 4.10]{MTT3}.

\subsection{The Galois group of $g_n$}
It is known that the Galois group of a polynomial of degree $m$ is generically the full symmetric group $S_m$. This also seems to be the case for the special family $g_n$ from \cref{def:g}. We use the technique mentioned above to verify this for a range of values of $n$.
We first demonstrate this technique using an example. 
\begin{expl} \label{expl:galois_for_g}
Let $n=3 \times 7 = 21$. Then 
\[ g_n(x) = x^8 + x^7 - 8x^6 - 7x^5 + 20x^4 + 14x^3 - 16x^2 - 6x + 2, \] 
is a polynomial of degree $m = 8$. Let $G = \Gal(\Q(g_n)|\Q)$.
Using SageMath, we see that $g_n$ is irreducible over $\F_{5}$. This shows that $G$ is a transitive subgroup of $S_m$.
Over $\F_{19}$, $g_n$ factors as the product of a linear factor and an irreducible factor of degree $m-1$:
\[ g_{n}(x) = (x + 8)  (x^7 + 12x^6 + 10x^5 + 8x^4 + 13x^3 + 5x^2 + x + 5) \pmod{19}.\]
This shows that $G$ contains an $(m-1)$-cycle.
Finally, over $\F_{7}$, we have the factorisation
\[ g_n(x) = (x^2 + x + 4) (x^3 + 4) (x^3 + 2x + 1) \pmod{7}, \]
which contains exactly one irreducible factor of degree $2$, and no other irreducible factor of even degree. This shows that $G$ contains a transposition.
Therefore $G = S_8$.
\end{expl}

\begin{question} \label{conj:galois_of_g}
Let $n=pq$ be a product of two distinct odd primes, $g_n$ be the polynomial from \cref{def:g}, and $m=\deg(g_n)$. Is the Galois group of $g_n$ isomorphic to $S_m$?
\end{question}
While this question seems hard to tackle, we find extensive computational evidence.
\begin{prop}
\label{prop:gal_gn_10000}
    The answer to \cref{conj:galois_of_g} is affirmative for all $n < 10^4$ such that $n$ is a product of two distinct odd primes $p, q$.
\end{prop}
\begin{proof}
    For each $n$, our GitHub repository \cite{fekete} lists prime triples $(q_1, q_2, q_3)$, such that the reduction of $g_n$ modulo $q_i$ has the desired factorization type as in \cref{expl:galois_for_g}.
\end{proof}

\subsection{The Galois group of $f_n$} 

Recall that the Fekete polynomials $f_n$ from \cref{defn:fn} are reciprocal polynomials whose trace polynomial is $g_n$ (\cref{def:g}). Let $m$ denote the degree of $g_n$. Then we have the following commutative diagram where the vertical maps are natural inclusions coming from the permutation action on roots.
\begin{equation}
    \begin{tikzcd}[column sep=small]
        1 \arrow[r] & \Gal(\Q(f_n)|\Q(g_n)) \arrow[r] \arrow[d, sloped, "\subseteq"] & \Gal(\Q(f_n)/\Q) \arrow[r] \arrow[d, sloped, "\subseteq"] & \Gal(\Q(g_n)/\Q) \arrow[r] \arrow[d, sloped, "\subseteq"] & 1\\
        1 \arrow[r] & (\Z/2)^m \arrow[r] & (\Z/2) \wr S_m \arrow[r] & S_m \arrow[r] & 1
    \end{tikzcd}
\end{equation}
Thus the Galois group of $f_n$ is naturally a subgroup of the wreath product $(\Z/2) \wr S_m$. Note that this wreath product is the semidirect product $(\Z/2)^m \rtimes S_m$, where $S_m$ acts by permuting the coordinates. It can be seen as a subgroup of $S_{2m}$ by \cite[Section 2]{davis1998probabilistic}.
We utilize \cite[Propositions 11.8, 11.11]{MTT4} to determine the Galois group of $f_n$ for a range of values of $n$. We first demonstrate their use with a couple of examples.

\begin{expl} \label{expl:galois_for_f}
Let $n=3 \times 7 = 21$. In \cref{expl:galois_for_g}, we showed that the Galois group of $g_n$ is $S_8$. 
By \cref{prop:discriminant_f}, the discriminant of $f_n$ is not a perfect square.  Further
\begin{align*} 
f_n(x) &= (x^2 + 12x + 1) \times (x^7 + 78x^6 + 173x^5 + 18x^4 + 119x^3 + 129x^2 + 107x + 9)  \\ 
& \times (x^7 + 138x^6 + 90x^5 + 215x^4 + 2x^3 + 221x^2 + 160x + 101) \pmod{227}.
\end{align*}
This factorisation over $\F_{227}$ contains exactly one irreducible factor of degree $2$, and no other irreducible factor of even degree.
By \cite[Proposition 11.11]{MTT4}, we conclude that the Galois group of $f_n$ is $(\Z/2)^8 \rtimes S_8$.
\end{expl}

\begin{expl}
Let $n = 5 \times 7 = 35$. Then $g_n$ is a polynomial of degree $11$ 
\[ g_n(x)= x^{11} - 11x^9 + 43x^7 + x^6 - 71x^5 - 5x^4 + 46x^3 + 4x^2 - 8x + 2.\]
As in \cref{expl:galois_for_g}, the factorisations of $g_n$ over $\F_{q_i}$ with $(q_1, q_2, q_3)=(29, 47, 31)$ yields that the Galois group of $g_n$ is $S_{11}$. Over $\F_{433}$, the polynomial $f_n$ factors as
\begin{align*}
f_n(x) = &(x + 97) \times (x + 125) \times (x^2 + 41x + 1) \times (x^4 + 124x^3 + 295x^2 + 124x + 1) \\ 
& \times (x^7 + 190x^6 + 62x^5 + 191x^4 + 406x^3 + 37x^2 + 393x + 313)  \\
& \times (x^7 + 289x^6 + 393x^5 + 76x^4 + 168x^3 + 50x^2 + 251x + 350) \pmod{433},
\end{align*}
with a unique irreducible factor each of degrees $2$ and $4$, and no other irreducible factor of even degree.
Remembering that now $\disc (f_n)$ is a perfect square by \cref{prop:discriminant_f}, and using \cite[Proposition 11.8]{MTT4}, we conclude that the Galois group of $f_n$ is $\ker(\Sigma) \rtimes S_{11}$ where $\Sigma: (\Z/2)^{11} \to \Z/2$ is the summation map.
\end{expl}

Thus we raise the following question, supported by extensive numerical evidence.
\begin{question} \label{conj:galois_f}
Let $n=pq$ be a product of two distinct odd primes, $f_n$ be the Fekete polynomial from \cref{defn:fn}, and $2m = \deg(f_n)$. Are the following statements true?
\begin{enumerate}
    \item Suppose $p \equiv 1 \pmod{q}$, or $p \not \equiv 1 \pmod{q}$ and $p,q \equiv 1 \pmod{4}$. Then the Galois group of $f_n$ is $(\Z/2)^m \rtimes S_m$.
    \item Suppose we are in the remaining case, i.e., $p \not \equiv 1 \pmod{q}$ and at least one of $p$ or $q$ is not congruent to $1 \pmod{4}$. Then the Galois group of $f_n$ is $\ker(\Sigma) \rtimes S_{m}$ where $\Sigma : (\Z/2)^{m} \to \Z/2$ is the summation map.
\end{enumerate}
\end{question}
\begin{prop}
\label{prop:gal_fn_10000}
    The answer to \cref{conj:galois_f} is affirmative for all $n < 10^4$ such that $n$ is a product of two distinct odd primes $p, q$.
\end{prop}
\begin{proof}
    By \cref{prop:gal_gn_10000}, the trace polynomial $g_n$ has Galois group isomorphic to $S_n$. For each $n$, our GitHub repository \cite{fekete} lists another prime $q_4$, such that the factorisation of $f_n$ modulo $q_4$ yields the desired conclusion, using \cite[Propositions 11.8, 11.11]{MTT4}.
\end{proof}

\section*{Code availability}

An open-source code repository for this work is available on GitHub \cite{fekete}.

\section*{Acknowledgments}

We thank Oleksiy Klurman, Andrew Granville, Bjorn Poonen and Kannan Soundararajan for helpful discussions and correspondence. The third named author would like to thank William Stein for his help with the platform Cocalc where our computations are based. 

\appendix

\section{Zeros on the unit circle}
\label{sec:zeros_on_unitcircle}

The complex zeros of classical Fekete polynomials $f_p(x) = \sum\limits_{a = 0}^{p-1} \chi(a) x^a$ for quadratic Dirichlet characters $\chi = \legendre{\cdot}{p}$ of prime conductor $p$ were studied in \cite{conrey2000zeros}. It was shown in \cite{conrey2000zeros} that at least half of the zeros of $f_p$ lie on the unit circle, and further that there exists a constant $1/2 < k_0 < 1$ such that the fraction of zeros of $f_p$ lying on the unit circle converges to $k_0$ as $p$ goes to infinity. In this section, we use the approach of \cite{conrey2000zeros} to analyze complex zeros of the Fekete polynomials $F_n$ corresponding to principal Dirichlet characters. We remark that since the coefficients of $F_n$ are either $0$ or $1$, the Erdos-Turan theorem implies that the roots of this polynomial are almost all clustered around the unit circle and equidistributed in angle (see \cite[Theorem 1]{erdos1950distribution} and \cite[Theorem 1.3]{granville2007distribution}). We thank Professor Kannan Soundararajan for pointing this out to us.

Let $H_n : \C \setminus (0,\infty) \rightarrow \C$ be the function defined by $H_n(z) = z^{-n/2}F_n(z)$ where we make a choice of the square root $z^{1/2}$. If $z = e^{2 \pi i t}$ we have
\begin{align*}
    H_n(z) & = z^{-n/2} \sum_{\substack{1 \leq a \leq n-1\\ \gcd(a,n)=1}} z^a = \sum_{\substack{1 \leq a \leq (n-1)/2\\ \gcd(a,n)=1}} (x^{a-n/2} + x^{n/2-a})\\
    & = \sum_{\substack{1 \leq a \leq (n-1)/2\\ \gcd(a,n)=1}} 2 \cos\left(\pi t (2a-n)\right).
\end{align*}
Let $\C_1 = \{z \in \C : \abs{z} = 1\}$ denote the unit circle in $\C$. Thus $H_n$ defines a continuous real valued function on $\C_1 \setminus \{1\}$. For $k \in \Z$, let $d_k$ denote $n/\gcd(n,k)$. By \cref{prop:value_at_zeta_d}, if $0 < k < n$ and $\zeta_n = e^{2\pi i/n}$, we have
\begin{align}
\label{eqn:values_of_Hn}
    H_n(\zeta_n^k) = \zeta_n^{-nk/2} F_n(\zeta_n^k) = \frac{(-1)^k \mu(d_k) \phi(n)}{\phi(d_k)}
\end{align}
If $k$ is such that $\gcd(n,k) = \gcd(n,k+1) = 1$, then $H_n$ changes sign on the arc from $\zeta_n^k$ to $\zeta_n^{k+1}$. Therefore, $H_n$ and hence $F_n$ must have a zero on this arc.

Let $\phi_1(n)$ denote the cardinality of the set $\{0 \leq a < n | \gcd(n,a) = \gcd(n,a+1) = 1\}$. Then $\phi_1 : \N \rightarrow \N$ is a multiplicative function by Chinese Remainder theorem, and $\phi_1(p^k) = p^k(1-2p^{-1})$. Thus we have the formula
$\phi_1(n) = n\prod_{p | n}\left(1-\frac{2}{p}\right)$. In summary, we have just proved the following. 
\begin{prop}
    $F_n$ has at least $\phi_1(n)$ roots on the unit circle where 
    \[ \phi_1(n) = n\prod_{p | n}\left(1-\frac{2}{p}\right) .\] 
\end{prop}


If $n=p$ or $n=2p$ where $p$ is a prime number then all factors of $F_n$, except $x$, are cyclotomic polynomials as explained in Examples \ref{expl:p_and_pk} and \ref{expl:2p}. The case where $n$ has exactly two odd prime factors is more interesting. Specifically, let us consider the following special case: we fix an odd prime $q$ and consider $n = pq$ for varying primes $p$. Then $\phi_1(n) \sim \left(1-\frac{2}{q}\right)n$ as $p \rightarrow \infty$. Therefore, the number of complex zeros of $F_n$ on the unit circle grows at least as fast as $k_0n$ in this limit, with $k_0 = 1-\frac{2}{q}$. It would be interesting to study this question for general $n.$

\bibliographystyle{abbrv}
\bibliography{references.bib}

\end{document}